\documentclass[11pt,a4paper]{article}

\usepackage[all]{xy}

\usepackage{amsmath, amssymb, amsthm}
\usepackage[textwidth=2.8cm]{todonotes}
\usepackage{yfonts}
\usepackage{mathrsfs}
\usepackage{tikz,tikz-cd}
\usepackage{color} 

\newtheorem{thm}{Theorem}   
\newtheorem{remark}{Remark}   
\newtheorem{lemma}{Lemma}   
\newtheorem{prop}{Proposition}   
\newtheorem{cor}{Corollary}

\newtheorem{definition}{Definition} 
 
\DeclareMathOperator{\Lip}{Lip}

\title{Lipschitz interpolative nonlinear ideal procedure}
\date{}
\author{M. A. S. SALEH \\  Department of Mathematics and Computer Applications, \\ College of Science, Al-Nahrain University, Baghdad, Iraq.\\ E-mail:  mas@sc.nahrainuniv.edu.iq}

\begin{document}

\maketitle

\begin{abstract}
We treat the general theory of nonlinear   ideals and extend as many notions as possible from the linear theory to the nonlinear theory. We define nonlinear ideals with special properties  which associate new non-linear ideals to given ones and establish several properties and characterizations of them. Building upon the results of U. Matter  we define a Lipschitz interpolative nonlinear ideal procedure  between metric spaces and Banach spaces and establish this class of Lipschitz operators is an injective Banach nonlinear ideal and show several standard basic properties for such class.  Extending the work of J. A. L\'{o}pez Molina  and E. A. S\'{a}nchez P\'{e}rez   we define   a Lipschitz $\left(p,\theta, q, \nu\right)$-dominated operators for $1\leq p, q <\infty$; $0\leq \theta, \nu< 1$ and establish several characterizations.  Afterwards we generalize a notion of Lipschitz interpolative nonlinear ideal procedure  between arbitrary metric spaces and prove its a nonlinear ideal. Finally, we  present certain  basic counter examples of  Lipschitz interpolative nonlinear ideal procedure between arbitrary metric spaces.  
\end{abstract}

 \section{Notations and Preliminaries}\label{Sec. 1}

We introduce concepts and notations that will be used in this article. The letters $E$, $F$ and $G$ will denote Banach spaces. The closed unit ball of a Banach space $E$ is denoted by $B_{E}$. The dual space of $E$ is  denoted by $E^{*}$. The class of all bounded linear operators between arbitrary Banach spaces will be denoted by $\mathfrak{L}$. The symbols $\mathbb{K}$ and $\mathbb{N}$ stand for the set of all scalar field and the set of all natural numbers, respectively. The symbols $W(B_{E^{*}})$ and $W(B_{X^{ \#}})$ stand for the set of all Borel probability measures defined on $B_{E^{*}}$ and $B_{X^{ \#}}$, respectively. The value of  $a$  at the element  $x$  is denoted by $\left\langle x, a\right\rangle$.  We put $E^{\text{inj}}:=\ell_{\infty}(B_{E^{*}})$ and $J_{E} x:=\left(\left\langle x, a\right\rangle\right)$ for $x\in E$. Clearly $J_{E}$ is a metric injection from $E$ into $E^{\text{inj}}$. Let $0<p<\infty$.  The Banach space of all absolutely $p$-summable sequences $\mathbf{x}=(x_j)_{j\in \mathbb{N}}$, where $x_j\in E$, is denoted by $\ell_{p}(E)$. We put 
$$\left\|\mathbf{x}\Big|\ell_{p}(E)\right\|=\Bigg[\sum\limits_{j=1}^{ \infty}\left\|x_j\right\|^{p}\Bigg]^{\frac{1}{p}}<\infty.$$

The Banach space of all weakly absolutely $p$-summable sequences $\mathbf{x}\subset E$, is denoted by $\ell_{p}^{w}(E)$. We put 
\begin{equation} 
\left\|\mathbf{x}\Big|\ell_{p}^{w}(E)\right\|=\sup\limits_{a\in B_{E^{*}}}\Bigg[\sum\limits_{j=1}^{\infty}\left|\left\langle x_j, a\right\rangle\right|^{p}\Bigg]^{\frac{1}{p}}.
\end{equation}  

For the triple sequence $(\sigma, x', x'')\subset\mathbb{R}\times X\times X$. We put  
 $$\left\|(\sigma,x',x'')\Big|\ell_p(\mathbb{R}\times X\times X)\right\|=\Bigg[\sum\limits_{j=1}^{\rm\infty}\left|\sigma_j\right|^{p} d_X(x'_j,x''_j)^{p}\Bigg]^{\frac{1}{p}}.$$

And  

$$\left\|(\sigma,x',x'')\Big|\ell_p^{L,w}(\mathbb{R}\times X\times X)\right\|=\sup\limits_{f\in B_{{X}^{\#}}}\Bigg[\sum\limits_{j=1}^{\infty}\left|\sigma_j\right|^{p}\left|\left\langle f,x'_j\right\rangle-\left\langle f,x''_j\right\rangle\right|^{p}\Bigg]^{\frac{1}{p}}.
$$

For $0\leq \theta <1$ and  $1\leq p<\infty$ we define
 $$\left\|  \mathbf{x} \Big|\delta_{p,\theta} (E)\right\| =\underset{\xi \in B_{E^{\ast }}}{\sup }\left( \underset{j=1}{\overset{\infty }{\sum }}\left(  \left|\left\langle  x_{j}  ,  \xi \right\rangle\right|^{1-\theta }\left\Vert x_{j}\right\Vert ^{\theta }\right) ^{\frac{p}{1-\theta }}\right) ^{\frac{1-\theta }{p}}.$$

Also for all sequences $(\sigma ,x',x'')\subset\mathbb{R}\times X\times X$,  we   define
$$\left\|(\sigma,x',x'')\Big|\delta_{p,\theta}^{L}(\mathbb{R}\times X\times X)\right\|=\sup\limits_{f\in B_{X^{\#}}}\left[ \sum\limits_{j=1}^{\infty }\left(
\left\vert \sigma _{j}\right\vert \left\vert f(x_{j}^{\prime
})-f(x_{j}^{\prime \prime })\right\vert ^{1-\theta }d_{X}(x_{j}^{\prime
},x_{j}^{\prime \prime })^{\theta }\right) ^{\frac{p}{1-\theta }}\right] ^{%
\frac{1-\theta }{p}}
$$

Recall that the definition of an operator ideal between arbitrary Banach spaces of A. Pietsch \cite {P07} and \cite{P87} is as follows. Suppose that, for every pair of Banach spaces $E$ and $F$, we are given a subset $\mathfrak{A}(E,F)$ of $\mathfrak{L}(E,F)$. The class  
$$\mathfrak{A}:=\bigcup_{E,F}\mathfrak{A}(E,F)$$
is said to be an operator ideal, or just an ideal, if the following conditions are satisfied:

\begin{enumerate}\label{Auto5}
	\item[$\bf (OI_0)$] $a^{*}\otimes e\in\mathfrak{A}(E,F)$ for $a^{*}\in E^{*}$ and $e\in F$.
	\item[$\bf (OI_1)$] $S + T\in\mathfrak{A}(E,F)$ for $S,\: T\in\mathfrak{A}(E,F)$.
	\item[$\bf (OI_2)$] $BTA\in\mathfrak{A}(E_{0},F_{0})$ for $A\in \mathfrak{L}(E_{0},E)$, $T\in\mathfrak{A}(E,F)$, and $B\in\mathfrak{L}(F,F_{0})$.
\end{enumerate}

Condition $\bf (OI_0)$ implies that $\mathfrak{A}$ contains nonzero operators. 

\begin{remark} 
The normed (Banach) operator ideal  is designated by $\left[\textfrak{A} , \mathbf{A} \right]$.
\end{remark} 

 \section{Introduction}\label{Sec. 2}  
 
One important example of operator ideals is the class of $p$-summing operators defined by A. Pietsch \cite {P78} as follow: A bounded  operator $T$ from $E$ into $F$ is called $p$-summing if and only if there is a constant $C\geq 0$  such that 
\begin{equation}\label{gfgfgfgfgjh20}
\left\| (T x_{j})_{j=1}^{m}\Big|\ell_p(F)\right\|\leq C\cdot\left\|\left( x_{j}\right)_{j=1}^{m}\Big|\ell_p^{w}(E)\right\|
\end{equation} 
for arbitrary sequence $\left(x_{j}\right)_{j=1}^{m}$ in $E$ and $m\in\mathbb{N}$. Let us denote by $\Pi_{p}(E,F)$ the class of all $p$-summing operators from $E$ into $F$ with $\pi_{p}(T)$ summing norm of $T$ is the infimum of such constants $C$. 

J. D. Farmer and W. B. Johnson \cite {J09} defined a true extension of the linear concept  of $p$-summing operators as follows: a Lipschitz operator $T\in \Lip(X,Y)$ is called Lipschitz $p$-summing map if there is a nonnegative constants $C$ such that for all $m\in\mathbb{N}$, any sequences $x'$, $x''$ in $X$ and $\lambda $ in $\mathbb{R}^{+}$, the inequality
$$\left\|(\lambda, Tx',Tx'')\Big|\ell_p(\mathbb{R}\times X\times X)\right\|\leq C\cdot\left\|(\lambda,x',x'')\Big|\ell_p^{L,w}(\mathbb{R}\times X\times X)\right\| 
$$
holds. Let us denote by $\Pi_{p}^{L}(X, Y)$ the class of all Lipschitz $p$-summing maps from $X$ into $Y$ with $\pi_{p}^{L}(T)$ Lipschitz summing norm of $T$ is the infimum of such constants $C$.

Jarchow and Matter \cite{JM88}    defined a general interpolation procedure for creating a new operator ideal between arbitrary Banach spaces. Also   U. Matter defined in his seminal paper \cite {Matter87}  a  new class of interpolative   ideal procedure  as follows: let $0\leq \theta< 1$ and $\left[\textfrak{A}, \mathbf{A}\right]$ be a normed  operator ideal. A bounded operator $T$ from $E$ into $F$ belongs to $\textfrak{A}_{\theta}(E, F)$ if there exist a Banach space $G$ and a bounded operator $S\in\textfrak{A} (E, G)$ such that 
\begin{equation}\label{prooor}
\left\|Tx|F\right\|\leq\left\|Sx |G\right\|^{1-\theta}\cdot  \left\|x\right\|^{\theta},\ \  \forall\: x  \in E.
\end{equation}
For each $T\in\textfrak{A} _{\theta}(E, F)$, we set 

\begin{equation}\label{poorooor}
\mathbf{A}_{\theta}(T):=\inf\mathbf{A}(S)^{1-\theta}
\end{equation}
where the infimum is taken over all bounded operators $S$ admitted in (\ref{prooor}).

\begin{prop} \cite {Matter87} 
$\left[\textfrak{A} _{\theta}, \mathbf{A}_{\theta}\right]$ is an injective complete quasinormed operator ideal.
\end{prop}

U. Matter \cite {Matter87} applied Inequality  (\ref{prooor}) to the ideal $\left[\Pi_p, \pi_p\right]$ of absolutely $p$-summing operators and obtained the injective
operator ideal $\left(\Pi_{p}\right)_{\theta}$ which  is complete  with respect to the ideal norm $\left(\pi_{p}\right)_{\theta}$ and established the fundamental theorem of $(p,\theta)$-summing operators for $1\leq p<\infty$ and $0\leq \theta <1$ as follows: 

\begin{thm}\cite {Matter87} 
Let $T$ be a bounded operator  from $E$ into $F$ and $C\geq 0$. The following are equivalent:
\begin{enumerate}
	\item $T\in\left(\Pi_{p}\right)_{\theta}(E, F)$.
	\item There exist a constant $C$ and a probability measure $\mu$ on $B_{E^{*}}$ such that 
	$$\left\|T x | F\right\| \leq C \cdot\left(\int\limits_{B_{E^{*}}}\left(\left|\left\langle x, x^{*}\right\rangle\right|^{1-\theta} \left\|x\right\|^{\theta}\right)^{\frac{p}{1-\theta}} d\mu(x^{*})\right)^{\frac{1-\theta}{p}}, \forall \;x\in E.$$
	 \item There exists a constant $C\geq 0$\ such that for any $\left( x_{j}\right) _{j=1}^{m}\subset E, $ and $m\in\mathbb{N}$ we have $$\left\Vert \left( Tx_{j}\right) _{j=1}^{m}\Big|\ell _{\frac{p}{1-\theta }}(F)\right\Vert \leq C\cdot \left\|\left( x_{j}\right)_{j=1}^{m} | \delta _{p\theta}(E)\right\|.$$
	
\end{enumerate}
In addition, $\left(\pi_{p}\right)_{\theta}(T)$ is the smallest number $C$ for which, respectively, (2) and (3) hold.
\end{thm}

Another example of operator ideals is the class of $(r,p, q)$-summing operators defined by A. Pietsch \cite [Sec. 17.1.1] {P78} as follows: Let $0<r, p, q\leq\infty$ and $\frac{1}{r}\leq\frac{1}{p}+ \frac{1}{q}$. A bounded  operator $T$ from $E$ into $F$ is called $(r,p, q)$-summing if there is a constant $C\geq 0$  such that 
\begin{equation}\label{20}
\left\| \left( \left\langle T x_{j}, b_{j}\right\rangle\right)_{j=1}^{m}\Big|\ell_r \right\Vert\leq C\cdot\left\|\left( x_{j}\right)_{j=1}^{m}\Big|\ell_p^{w}(E)\right\| \left\|\left( b_{j}\right)_{j=1}^{m}\Big|\ell_q^{w}(F^{*})\right\|
\end{equation} 
for arbitrary sequence $\left(x_{j}\right)_{j=1}^{m}$ in $E$, $\left(b_{j}\right)_{j=1}^{m}$ in $F^{*}$ and $m\in\mathbb{N}$. Let us denote by $\mathfrak{P}_{(r, p, q)}(E,F)$ the class of all $(r, p, q)$-summing operators from $E$ into $F$ with $\textbf{P}_{(r, p, q)}(T)$ summing norm of $T$ is the infimum of such constants $C$.

\begin{prop} \cite [Sec. 17.1.2] {P78}
$\left[\mathfrak{P}_{(r, p, q)}, \textbf{P}_{(r, p, q)}\right]$  is a normed operator ideal.
\end{prop}

Let $0<p, q\leq\infty$. A. Pietsch \cite {P78} is also defined $(p, q)$-dominated operator as follows: A bounded  operator $T$ from $E$ into $F$ is called $(p, q)$-dominated if it belongs to the quasi-normed ideal 
$$\left[\mathcal{D}_{\left(p, q\right)}, D_{\left(p, q\right)}\right]:=\left[\mathfrak{P}_{(r, p, q)}, \textbf{P}_{(r, p, q)}\right],$$  
where $\frac{1}{r}=\frac{1}{p} + \frac{1}{q}$. For a special case, if $q=\infty$, then $\left[\mathcal{D}_{\left(p,  \infty\right)}, D_{\left(p,  \infty\right)}\right]:=\left[\Pi_{p}, \pi_{p}\right]$. 

J.A. L\'{o}pez Molina  and E. A. S\'{a}nchez P\'{e}rez \cite {ms93} established the important characteristic of $(p, q)$-dominated operator as follows.

\begin{prop} \cite {ms93}
Let $E$ and $F$ be Banach spaces and $T\in\mathfrak{L}(E,F)$. The following are equivalent:
 
 \begin{enumerate}
	 \item $T\in\mathcal{D}_{\left(p, q\right)}(E,F)$.
	
	 \item There exist a Banach spaces $G$ and $H$, bounded operators  $S_{1}\in\Pi_{p}(E, G)$ and $S_{2}\in\Pi_{q}(F^{*}, H)$ and $C>0$ such that 
\begin{equation} 
 \left|\left\langle Tx, b\right\rangle\right|\leq C \left\|S_{1} x\right\| \left\|S_{2} x\right\|, \ \   \forall x \in E, \forall \;b\in F^{*}.
\end{equation}
\end{enumerate}
\end{prop}

A general example of $(p, q)$-dominated operators is also defined by J.A. L\'{o}pez Molina  and E. A. S\'{a}nchez P\'{e}rez \cite {ms93}  as follows: Let $1\leq p, q <\infty$ and $0\leq \theta, \nu< 1$ such that $\frac{1}{r}+\frac{1-\theta}{p}+\frac{1-\nu}{q}=1$ with $1\leq r <\infty$.  A bounded operator $T$ from $E$ to $F$ is called   $\left(p,\theta, q, \nu\right)$-dominated if there exist a Banach spaces $G$ and $H$, a bounded operator $S\in\Pi_{p} (E,G)$, a bounded operator $R\in\Pi_{q}(F^{*},H)$ and a positive constant $C$ such that 
\begin{equation}\label{lastrrrwagen1} 
\left|\left\langle Tx, b^{*}\right\rangle\right|\leq C\cdot  \left\|x\right\|^{\theta}\left\|Sx | G\right\|^{1-\theta}\left\|b^{*}\right\|^{\nu}\left\|R(b^{*})|H\right\|^{1-\nu} 
\end{equation}
for arbitrary finite sequences $\textbf{x}$ in $X$ and $b^{*}\subset F^{*}$. 

Let us denote by $\mathcal{D}_{\left(p, \theta, q, \nu\right)} (E,F)$ the class of all   $\left(p, \theta, q, \nu\right)$-dominated operators from $E$ to $F$ with $$D_{\left(p, \theta, q, \nu\right)} (T)=\inf\left\{C \cdot\pi_{p} (S)^{1-\theta}\cdot \pi_{q}(R)^{1-\nu}\right\},$$
where the infimum is taken over all bounded operators $S$ and $R$  and constant $C$ admitted in (\ref{lastrrrwagen1}). They also established an important characteristic of $\left(p,\theta, q, \nu\right)$-dominated operator as follows.

\begin{thm}
Let $E$ and $F$ be Banach spaces and $T\in\mathfrak{L}(E,F)$. The following are equivalent:

\begin{enumerate}
	\item [$\bf (1)$] $T\in\mathcal{D}_{\left(p, \theta, q, \nu\right)} (E,F)$.
	
	\item [$\bf (2)$] There is a constant $C\geq 0$ and regular probabilities $\mu$ and $\tau$ on $B_{E^{*}}$ and $B_{F^{**}}$, respectively such that for every $x$ in $X$ and $b^{*}$ in $F^{*}$ the following inequality holds
$$\left|\left\langle Tx, b^{*}\right\rangle\right|\leq C  \cdot\left[\int\limits_{B_{E^{*}}}\left(\left| \left\langle x, a\right\rangle\right|^{1-\theta}  \left\|x\right\|^{\theta}\right)^\frac{p}{1-\theta} d\mu(a)\right]^\frac{1-\theta}{p} 
\cdot\left[\int\limits_{B_{F^{**}}}\left(\left|\left\langle b^{*}, b^{**}\right\rangle\right|^{1-\nu} \left\|b^{*}\right\|^{\nu}\right)^\frac{q}{1-\nu} d\tau(b^{**})\right]^\frac{1-\nu}{q}.  
$$
\item[$\bf (3)$]  There exists a constant $C\geq 0$ such that for every finite sequences $\textbf{x} $ in $X$  and $b^{*}\subset F^{*}$ the inequality
\begin{equation}
\left\| \left\langle Tx, b^{*}\right\rangle|\ell_{r'}\right\|\leq C\cdot\left\|\textbf{x}\Big|\delta_{p,\theta} (E)\right\|\left\| b^{*} \Big|\delta_{q,\nu}(F^{*})\right\|
\end{equation}
holds. 
\item[$\bf (4)$] There are a Banach space $G$, a bounded operator $A\in\left(\Pi_{p}\right)_{\theta}(X,G)$ and a bounded operator $B\in\mathfrak{L}(E,F)$ such that $B^{*}\in\left(\Pi_{q}\right)_{\nu}^{dual}(F^{*}, G^{*})$ and $T=BA$.
\end{enumerate}
In this case, $D_{\left(p, \theta, q, \nu\right)}$ is equal to the infimum of such constants $C$ in either $\bf (2)$, or $\bf (3)$.

\end{thm}


We now describe the contents of this paper. In Section \ref{Sec. 1}, we introduce notations and preliminaries that will be used in this article. In Section \ref{Sec. 2}, we first present preliminaries of special cases of those operators that map weakly (Lipschitz) $p$-summable sequences  in   arbitrary Banach (metric) space  into strongly (Lipschitz) $p$-summable  ones in  Banach (metric) space  these operators are called (Lipschitz) $p$-summing operators defined  by A. Pietsch \cite {P78},   J. D. Farmer and W. B. Johnson \cite {J09}, respectively. Jarchow and Matter \cite{JM88}  defined a general interpolation procedure to create  a new  ideal from   given  ideals and  U. Matter defined a  new class of interpolative ideal procedure in his seminal paper \cite {Matter87}. He established the fundamental  characterize result of $(p,\theta)$-summing operators for $1\leq p<\infty$ and $0\leq \theta <1$. For $0<p, q\leq\infty$. A. Pietsch \cite {P78}   defined $(p, q)$-dominated operator between arbitrary Banach spaces. J. A. L\'{o}pez Molina  and E. A. S\'{a}nchez P\'{e}rez \cite {ms93} established the fundamental  characterize  of $(p, q)$-dominated operator. Afterwards  a general example of $(p, q)$-dominated operators is also defined by J.A. L\'{o}pez Molina  and E. A. S\'{a}nchez P\'{e}rez \cite {ms93}. This class of operators is called   $\left(p,\theta, q, \nu\right)$-dominated for $1\leq p, q <\infty$ and $0\leq \theta, \nu< 1$ such that $\frac{1}{r}+\frac{1-\theta}{p}+\frac{1-\nu}{q}=1$ with $1\leq r <\infty$. They proved an important  characterize of $\left(p,\theta, q, \nu\right)$-dominated operator. In Section \ref{Sec. 3}, we treat the general theory of nonlinear operator ideals. The basic idea here is to extend as many notions as possible from the linear theory to the nonlinear theory. Therefore, we start by recalling the fundamental concepts of an operator ideal defined by A. Pietsch \cite{P78}, see also \cite{P07}. 
  Then, we introduce the corresponding definitions  for nonlinear 
operator ideals in the version close to that of A. Jim{\'e}nez-Vargas, J. M. Sepulcre, and Mois{\'e}s Villegas-Vallecillos \cite{mjam15}. Afterwards, we define nonlinear ideals with special properties which associate new non-linear ideals to given ones. Again, this is parallel to the linear theory. For $0<p\leq 1$ we also define a Lipschitz $p$-norm on nonlinear  ideal and prove that the injective hull $\textfrak{A}^{L}_{\text{inj}}$ is a $p$-normed nonlinear ideal. We generalize U. Matter's interpolative ideal procedure for its nonlinear (Lipschitz) version  between metric spaces and Banach spaces and establish these class of operators is an injective Banach nonlinear ideal as well as we show several basic properties for such class.   Extending the work of J. A. L\'{o}pez Molina  and E. A. S\'{a}nchez P\'{e}rez   we define   a Lipschitz $\left(p,\theta, q, \nu\right)$-dominated operators for $1\leq p, q <\infty$ and $0\leq \theta, \nu< 1$ such that $\frac{1}{r}+\frac{1-\theta}{p}+\frac{1-\nu}{q}=1$ with $1\leq r <\infty$ and establish several characterizations  analogous to linear case of \cite {ms93} and prove that the class of Lipschitz $\left(p,\theta, q, \nu\right)$-dominated operators is a Banach nonlinear ideal under the Lipschitz  $\left(p,\theta, q, \nu\right)$-norm.  In Section \ref{Sec. 4},  we define nonlinear operator ideal concept between arbitrary metric spaces. It is also in the version close to that defined in  \cite {mjam15}.  We generalize a notion of Lipschitz interpolative nonlinear ideal procedure  between arbitrary metric spaces and prove its a nonlinear ideal. Finally, we  present certain  basic counter examples of  Lipschitz interpolative nonlinear ideal procedure between arbitrary metric spaces.

\section{\bf Nonlinear  ideals between arbitrary metric spaces and Banach spaces} \label{Sec. 3}

\begin{definition}\label{A7777}

Suppose that, for every pair of metric spaces $X$ and Banach spaces $F$, we are given a subset $\textfrak{A}^{L}(X,F)$ of $\Lip(X,F)$. The class  
$$\textfrak{A}^{L}:=\bigcup_{X,F}\textfrak{A}^{L}(X,F)$$
is said to be a complete $p$-normed  (Banach) nonlinear ideal $\left(0<p\leq 1\right)$, if the following   conditions are satisfied:

\begin{enumerate}

	\item[$\bf (\widetilde{PNOI_0})$]  $g\boxdot e\in\textfrak{A}^{L}(X,F)$ and $\mathbf{A}^{L}\left(g\boxdot e\right)=\Lip(g)\cdot\left\|e\right\|$ for $g\in X^{\#}$ and $e\in F$.
	
	\item[$\bf (\widetilde{PNOI_1})$]  $S + T\in\textfrak{A}^{L}(X,F)$  and the $p$-triangle inequality holds:
	$$\mathbf{A}^{L}\left(S + T\right)^{p}\leq\mathbf{A}^{L}(S)^{p} + \mathbf{A}^{L}(T)^{p} \ \text{for} \  S,\: T\in\textfrak{A}^{L}(X,F).$$ 
	
	\item[$\bf (\widetilde{PNOI_2})$] $BTA\in\textfrak{A}^{L}(X_{0},F_{0})$ and $\mathbf{A}^{L}\left(BTA\right)\leq\left\|B\right\|\mathbf{A}^{L}(T)\: \Lip(A)$ for $A\in \Lip(X_{0},X)$, $T\in\textfrak{A}^{L}(X,F)$, and $B\in\mathfrak{L}(F,F_{0})$.

		\item[$\bf (\widetilde{PNOI_3})$] All linear spaces $\textfrak{A}^{L}(X,F)$ are complete, where $\mathbf{A}^{L}$ is called a Lipschitz $p$-norm from $\textfrak{A}^{L}$ to $\mathbb{R}^{+}$.	
	
\end{enumerate}

\end{definition}

\begin{remark}\label{soon}  

\begin{description}
	\item[$\bf (1)$] If $p=1$, then $\mathbf{A}^{L}$ is simply called a Lipschitz norm and $\left[\textfrak{A}^{L}, \mathbf{A}^{L}\right]$ is said to be a Banach nonlinear ideal. 
	
	\item[$\bf (2)$] If $\left[\textfrak{A}^{L}, \mathbf{A}^{L}\right]$ be a normed nonlinear ideal, then $\textfrak{A}^{L}\left(X, \mathbb{R}\right)=X^{\#}$ with $\Lip(g)=\mathbf{A}^{L}(g),\forall \ \ g\in X^{\#}$.
\end{description}

\end{remark}    

\begin{prop}\label{Auto4}
Let $\textfrak{A}^{L}$ be a nonlinear ideal. Then all components $\textfrak{A}^{L}(X,F)$ are linear spaces.
\end{prop}

\begin{proof}
By the   condition of $\bf (\widetilde{PNOI_1})$ it remains to show that $T\in\textfrak{A}^{L}(X,F)$ and $\lambda\in\mathbb{K}$ imply $\lambda\cdot T\in\textfrak{A}^{L}(X,F)$. This follows from $\lambda\cdot T=\left(\lambda\cdot I_{F}\right)\circ T\circ I_{X}$ and $\bf (\widetilde{PNOI_2})$. \\
\end{proof}

\begin{prop} 
If $\left[\textfrak{A}^{L}, \mathbf{A}^{L}\right]$ be a normed nonlinear ideal, then $\Lip(T)\leq\mathbf{A}^{L}(T)$ for all $T\in\textfrak{A}^{L}$.  
\end{prop}

\begin{proof}
Let $T$ be an arbitrary Lipschitz operator in $\textfrak{A}^{L}(X,F)$. 
\begin{align}
\Lip(T)=\left\|T^{\#}_{|_{{F}^{*}}}\right\|&=\sup\left\{\Lip(T^{\#} b^{*}) : b^{*}\in B_{{F}^{*}}\right\}\nonumber \\
&=\sup\left\{\Lip\left(b^{*}\circ T\right) : b^{*}\in B_{{F}^{*}}\right\} \nonumber 
\end{align}
Now from Remark \ref{soon} we have $\Lip(b^{*}\circ T)=\mathbf{A}^{L}(b^{*}\circ T)$ for $b^{*}\in {F}^{*}$. It follows $$\Lip(T)=\sup\left\{\mathbf{A}^{L}(b^{*}\circ T) : b^{*}\in B_{{F}^{*}}\right\}\leq\mathbf{A}^{L}(T).$$
\end{proof}

\subsection{\bf Nonlinear Ideals with Special Properties}\label{mAuto1818}

\subsubsection{Lipschitz Procedures}

A rule $$\text{new}: \mathfrak{A}\longrightarrow \mathfrak{A}^{L}_{new}$$

which defines a new nonlinear ideal   $\mathfrak{A}^{L}_{new}$ for every ideal $\mathfrak{A}$ is called a  Lipschitz semi-procedure. A rule $$\text{new} : \textfrak{A}^{L}\longrightarrow \textfrak{A}^{L}_{new}$$

which defines a new nonlinear ideal  $\textfrak{A}^{L}_{new}$ for every nonlinear ideal    $\textfrak{A}^{L}$ is called a  Lipschitz procedure.

\begin{remark}
We now define the following special properties:

\begin{enumerate}

	\item[$\bf (M')$] If $\textfrak{A}^{L}\subseteq\textfrak{B}^{L}$, then $\textfrak{A}^{L}_{new}\subseteq\textfrak{B}^{L}_{new}$  (strong monotony).
	
	\item[$\bf (M'')$] If $\mathfrak{A}\subseteq\mathfrak{B}$, then $\mathfrak{A}^{L}_{new}\subseteq\mathfrak{B}^{L}_{new}$  (monotony).
	
	\item[$\bf (I)$] $\left(\textfrak{A}^{L}_{new}\right)_{new}=\textfrak{A}^{L}_{new}$ for all $\textfrak{A}^{L}$  (idempotence).
	
\end{enumerate}

A strong monotone and idempotent Lipschitz procedure is called a Lipschitz hull  procedure if $\textfrak{A}^{L}\subseteq\textfrak{A}^{L}_{new}$ for all nonlinear ideals. 

\end{remark}

\subsubsection{Closed Nonlinear Ideals}\label{Auto1818}
Let $\textfrak{A}^{L}$ be a nonlinear ideal. A Lipschitz operator $T\in \Lip(X,F)$ belongs to the closure $\textfrak{A}^{L}_{clos}$ if there are $T_{1}, T_{2}, T_{3},\cdots\in\textfrak{A}^{L}(X, F)$ with $\lim\limits_{n} \Lip\left(T - T_{n}\right)=0$. It is not difficult to prove the following result.

\begin{prop}
$\textfrak{A}^{L}_{clos}$ is a nonlinear ideal.        
\end{prop}

The following statement is evident.

\begin{prop}
The rule $$\text{clos}: \textfrak{A}^{L}\longrightarrow\textfrak{A}^{L}_{clos}$$
is a hull Lipschitz procedure. 
\end{prop}

\begin{definition}
The nonlinear ideal  $\textfrak{A}^{L}$ is called closed if $\textfrak{A}^{L}=\textfrak{A}^{L}_{clos}$.
\end{definition} 
\begin{prop}
  Let $\textfrak{G}^{L}$ be a Lipschitz approximable nonlinear ideal. Then $\textfrak{G}^{L}$ is the smallest closed nonlinear ideal. 
\end{prop}

\begin{proof}
By the definition of Lipschitz approximable operators  in \cite{JAJM14} we have $\textfrak{G}^{L}=\textfrak{F}^{L}_{clos}$.  
Hence $\textfrak{G}^{L}$ is closed.
Let $\textfrak{A}^{L}$ be a closed nonlinear ideal. Since $\textfrak{F}^{L}$ is the smallest nonlinear ideal, we obtain from the monotonicity of the closure procedure
$$ \textfrak{G}^{L}=\textfrak{F}^{L}_{clos} \subseteq \textfrak{A}^{L}_{clos} = \textfrak{A}^{L}.$$
\end{proof} 
\subsubsection{Dual Nonlinear Ideals}\label{Auto19}
Let $\mathfrak{A}$ be an ideal. A Lipschitz operator $T\in \Lip(X,F)$ belongs to the Lipschitz dual ideal $\mathfrak{A}^{L}_{dual}$ if $T^{\#}_{|_{{F}^{*}}}\in\mathfrak{A}(F^{*}, X^{\#})$.        

\begin{lemma}\label{Auto244}
Let $T$ in $\textfrak{F}^{L}(X, F)$ with $T=\sum\limits_{j=1}^{m} g_{j}\boxdot e_{j}$. Then $T^{\#}_{|_{{F}^{*}}}=\sum\limits_{j=1}^{m} \hat{e}_{j}\otimes g_{j}$, where $e\longmapsto \hat{e}$ is the natural embedding of the space $F$ into its second dual $F^{**}$.
\end{lemma}

\begin{proof}
We have $Tx=\sum\limits_{j=1}^{m} g_{j}(x) e_{j}$ for $x\in X$. So for $b^{*}\in F^{*}$,
$$\left\langle T^{\#}_{|_{{F}^{*}}} b^{*},x\right\rangle_{(X^{\#},X)}=\left\langle b^{*}, Tx\right\rangle_{(F^{*},F)}=\sum\limits_{j=1}^{m}g_{j}(x) b^{*} (e_{j}).$$ Hence $T^{\#}_{|_{{F}^{*}}} b^{*}=\sum\limits_{j=1}^{m} b^{*}(e_{j}) g_{j}$. This proves the statement for $T^{\#}_{|_{{F}^{*}}}$. 
\end{proof}

\begin{lemma}\label{Auto16}
Let $T, S\in \Lip(X,F)$, $A\in \Lip(X_{0},X)$, and $B\in\mathfrak{L}(F,F_{0})$. Then
\begin{enumerate}
  	\item $\left(T+S\right)^{\#}_{|_{{F}^{*}}}=T^{\#}_{|_{{F}^{*}}} + S^{\#}_{|_{{F}^{*}}}$.
	  \item $\left(BTA\right)^{\#}_{|_{F^{*}_{0}}}=A^{\#}T^{\#}_{|_{{F}^{*}}}B^{*}$.
\end{enumerate}

 \end{lemma}

\begin{proof}
For  $b^{*}\in F^{*}$ and $x\in X$, we have
\begin{align}
\left\langle \left(T+S\right)^{\#}_{|_{{F}^{*}}} b^{*},x\right\rangle_{(X^{\#},X)}&=\left\langle b^{*}, (T+S)x\right\rangle_{(F^{*},F)}=\left\langle b^{*}, Tx+Sx\right\rangle_{(F^{*},F)}  \nonumber \\
&=\left\langle b^{*}, Tx\right\rangle_{(F^{*},F)} + \left\langle b^{*}, Sx\right\rangle_{(F^{*},F)}\nonumber \\
&=\left\langle T^{\#}_{|_{{F}^{*}}} b^{*},x\right\rangle_{(X^{\#},X)} + \left\langle S^{\#}_{|_{{F}^{*}}} b^{*},x\right\rangle_{(X^{\#},X)}. \nonumber
\end{align}
Hence $\left(T+S\right)^{\#}_{|_{{F}^{*}}}=T^{\#}_{|_{{F}^{*}}} + S^{\#}_{|_{{F}^{*}}}$. For  $b^{*}_{0}\in F^{*}$ and $x_{0}\in X_{0}$, we have
\begin{align}
\left\langle b_{0}^{*}, BTA (x_{0})\right\rangle_{(F_{0}^{*},F_{0})}&=\left\langle b_{0}^{*}, B(TA x_{0})\right\rangle_{(F_{0}^{*},F_{0})}=\left\langle B^{*} b_{0}^{*}, T(A x_{0})\right\rangle_{(F^{*},F)}  \nonumber \\
&=\left\langle T^{\#}_{|_{{F}^{*}}} B^{*} b_{0}^{*}, A x_{0}\right\rangle_{(X^{\#}, X)}=\left\langle A^{\#} T^{\#}_{|_{{F}^{*}}} B^{*} b_{0}^{*}, x_{0}\right\rangle_{(X_{0}^{\#}, X_{0})}. \nonumber 
\end{align}
But also $\left\langle b_{0}^{*}, BTA (x)\right\rangle_{(F_{0}^{*},F_{0})}=\left\langle \left(BTA\right)^{\#}_{|_{F^{*}_{0}}} b_{0}^{*}, x_{0}\right\rangle_{(X_{0}^{\#}, X_{0})}$. Therefore $\left(BTA\right)^{\#}_{|_{F^{*}_{0}}}=A^{\#}T^{\#}_{|_{{F}^{*}}}B^{*}$. 
\end{proof}

\begin{prop}\label{Auto26}
$\mathfrak{A}^{L}_{dual}$ is a nonlinear ideal. 
\end{prop}

\begin{proof}
The algebraic condition $\bf (\widetilde{PNOI_0})$ is satisfied, from Lemma \ref{Auto244} we obtain $\left(g\boxdot e\right)^{\#}_{|_{{F}^{*}}}=\hat{e}\otimes g\in\mathfrak{A}(F^{*}, X^{\#})$. To prove the algebraic condition $\bf (\widetilde{PNOI_1})$, let $T$ and $S$ in $\mathfrak{A}^{L}_{dual}(X,F)$. Let $T^{\#}_{|_{{F}^{*}}}$ and $S^{\#}_{|_{{F}^{*}}}$ in $\mathfrak{A}(F^{*}, X^{\#})$, from Lemma \ref {Auto16} we have $\left(T+S\right)^{\#}_{|_{{F}^{*}}}=T^{\#}_{|_{{F}^{*}}} + S^{\#}_{|_{{F}^{*}}}\in\mathfrak{A}(F^{*}, X^{\#})$. Let $A\in \Lip(X_{0},X)$, $T\in\mathfrak{A}^{L}_{dual}(X,F)$, and $B\in\mathfrak{L}(F,F_{0})$. Also from Lemma \ref {Auto16} we have $\left(BTA\right)^{\#}_{|_{F^{*}_{0}}}=A^{\#}T^{\#}_{|_{{F}^{*}}}B^{*}\in\mathfrak{A}(F_{0}^{*}, X_{0}^{\#})$, hence the algebraic condition $\bf (\widetilde{PNOI_2})$ is satisfied.  
\end{proof}

The following proposition is obvious.
\begin{prop}\label{tensor2}
The rule $$dual: \mathfrak{A}\longrightarrow \left(\mathfrak{A}\right)^{L}_{dual}$$
is a monotone Lipschitz procedure. 
\end{prop}


\begin{prop}\label{peep98}
Let $\textfrak{F}^{L}$ be a nonlinear ideal of   Lipschitz finite rank operators, $\mathfrak{F}$ be an ideal of  finite rank operators and $(\mathfrak{F})^{L}_{dual}$ be a  semi-Lipschitz procedure. Then  $\textfrak{F}^{L}=(\mathfrak{F})^{L}_{dual}$.
\end{prop}

\begin{proof}
Let $T\in\textfrak{F}^{L}(X, F)$, then $T$ can be represented in the form $\sum\limits_{j=1}^{m} g_{j}\boxdot e_{j}$. From Lemma \ref{Auto244} and $E^{*}\otimes F\equiv\mathfrak{F}(E, F)$  we have $T^{\#}_{|_{{F}^{*}}}=\sum\limits_{j=1}^{m}\hat{e}_{j}\otimes g_{j}\in F^{**}\otimes X^{\#}\equiv\mathfrak{F}(F^{*}, X^{\#})$. Hence $T\in\mathfrak{F}^{L}_{dual}(X, F)$. 

Let $T\in\mathfrak{F}^{L}_{dual}(X, F)$ then $T^{\#}_{|_{{F}^{*}}}\in\mathfrak{F}(F^{*}, X^{\#})$ hence $T^{\#}_{|_{{F}^{*}}}$ can be represented in the form $\sum\limits_{j=1}^{m} \hat{e}_{j}\otimes g_{j}$. For  $b^{*}\in F^{*}$ and $x\in X$, we have
\begin{align}
\left\langle b^{*}, T x\right\rangle_{(F^{*},F)}&=\left\langle T^{\#}_{|_{{F}^{*}}} b^{*}, x\right\rangle_{(X^{\#}, X)}=\left\langle \sum\limits_{j=1}^{m} \hat{e}_{j}\otimes g_{j}\  (b^{*}), x\right\rangle_{(X^{\#}, X)}=\left\langle \sum\limits_{j=1}^{m} \hat{e}_{j}(b^{*})\cdot g_{j}\  , x\right\rangle_{(X^{\#}, X)}  \nonumber \\
&=\left\langle \sum\limits_{j=1}^{m} b^{*}({e}_{j})\cdot g_{j}\  , x\right\rangle_{(X^{\#}, X)}=\sum\limits_{j=1}^{m} g_{j}(x)\cdot b^{*}(e_{j})=\left\langle b^{*}, \sum\limits_{j=1}^{m} g_{j}\boxdot e_{j} \ (x)\right\rangle_{(F^{*},F)}. \nonumber
\end{align}
Hence $T=\sum\limits_{j=1}^{m} g_{j}\boxdot e_{j}\in\textfrak{F}^{L}(X, F)$.  
\end{proof}

\subsubsection{Injective Nonlinear Ideals} 
Let $\textfrak{A}^{L}$ be a nonlinear ideal. A Lipschitz operator $T\in \Lip(X,F)$ belongs to the injective hull $\textfrak{A}^{L}_{inj}$ if $J_{F}T\in\textfrak{A}^{L}(X, F^{inj})$.

\begin{prop}\label{Auto29}
$\textfrak{A}^{L}_{inj}$ is a nonlinear ideal.    
\end{prop}

\begin{proof}
The algebraic condition $\bf (\widetilde{PNOI_0})$ is satisfied, since $g\boxdot e\in\textfrak{A}^{L}(X, F)$ and using nonlinear composition ideal property we have $J_{F}(g\boxdot e)\in\textfrak{A}^{L}(X, F^{inj})$. To prove the algebraic condition $\bf (\widetilde{PNOI_1})$, let $T$ and $S$ in $\textfrak{A}^{L}_{inj}(X,F)$. Then $J_{F}T$ and $J_{F}S$ in $\textfrak{A}^{L}(X, F^{inj})$, we have $J_{F}(T+S)=J_{F}T + J_{F}S\in\textfrak{A}^{L}(X, F^{inj})$. Let $A\in \Lip(X_{0},X)$, $T\in\textfrak{A}^{L}_{inj}(X, F)$, and $B\in\mathfrak{L}(F,F_{0})$. Since $F_{0}^{inj}$ has the extension property, there exists $B^{inj}\in\mathfrak{L}(F^{inj},F_{0}^{inj})$ such that 
$$
\begin{tikzcd}[row sep=5.0em, column sep=5.0em]
X   \arrow{r}{T}                & F   \arrow{r}{J_F} \arrow{d}{B} & F^{inj} \arrow{d}{B^{inj}} \\
X_0 \arrow{r}{BTA} \arrow{u}{A} & F_0 \arrow{r}{J_{F_0}}          & F_0^{inj}                  \\
\end{tikzcd}
\vspace{-50pt}
$$
Consequently $J_{F_{0}}\left(BTA\right)=B^{inj}\left(J_{F}T\right)A\in\textfrak{A}^{L}$, hence the algebraic condition $\bf (\widetilde{PNOI_2})$ is satisfied.  
\end{proof}

\begin{prop} 
The rule $$\text{inj}: \textfrak{A}^{L}\longrightarrow\textfrak{A}^{L}_{inj}$$
is a hull Lipschitz procedure. 
\end{prop}

\begin{proof}
The property $\bf (M')$ is obvious. To show the idempotence, let $T\in \Lip(X,F)$ belong to $\left(\textfrak{A}^{L}_{inj}\right)_{inj}$. Then $J_{F}T\in\textfrak{A}^{L}_{inj}(X, F^{inj})$, and the preceding lemma implies $J_{F}T\in\textfrak{A}^{L}(X, F^{inj})$. Consequently $T\in\textfrak{A}^{L}_{inj}(X, F)$. Thus $\left(\textfrak{A}^{L}_{inj}\right)_{inj}\subseteq\textfrak{A}^{L}_{inj}$. The converse inclusion is trivial.  
\end{proof}

\begin{lemma}\label{Auto21}
Let $F$ be a Banach space possessing the extension property. Then $\textfrak{A}^{L}(X, F)=\textfrak{A}^{L}_{inj}(X, F)$.
\end{lemma}

\begin{proof}
By hypothesis there exists $B\in\mathfrak{L}(F^{inj},F)$ such that $BJ_{F}=I_{F}$. Therefore $T\in\textfrak{A}^{L}_{inj}(X, F)$ implies that $T=B\left(J_{F}T\right)\in\textfrak{A}^{L}(X, F)$. This proves that $\textfrak{A}^{L}_{inj}\subseteq\textfrak{A}^{L}$. The converse inclusion is obvious.  
\end{proof}

\begin{prop}\label{Auto 14}
The rule $$inj: \textfrak{A}^{L}\longrightarrow\textfrak{A}^{L}_{inj}$$
is a hull Lipschitz procedure. 
\end{prop}

\begin{proof}
The property $\bf (M')$ is obvious. To show the idempotence, let $T\in \Lip(X,F)$ belong to $\left(\textfrak{A}^{L}_{inj}\right)_{inj}$. Then $J_{F}T\in\textfrak{A}^{L}_{inj}(X, F^{inj})$, and the preceding lemma implies $J_{F}T\in\textfrak{A}^{L}(X, F^{inj})$. Consequently $T\in\textfrak{A}^{L}_{inj}(X, F)$. Thus $\left(\textfrak{A}^{L}_{inj}\right)_{inj}\subseteq\textfrak{A}^{L}_{inj}$. The converse inclusion is trivial.  
\end{proof}
 
\subsection{Minimal Nonlinear Ideals}\label{Auto3030}
Let $\mathfrak{A}$ be an ideal. A Lipschitz operator $T\in \Lip(X,F)$ belongs to the associated minimal ideal $(\mathfrak{A})^{L}_{min}$ if $T=BT_{0}A$, where $B\in\mathfrak{G}(F_{0}, F)$, $T_{0}\in\mathfrak{A}(G_{0}, F_{0})$, and $A\in\textfrak{G}^{L}(X, G_{0})$. In the other words $(\mathfrak{A})^{L}_{min}:=\mathfrak{G}\circ\mathfrak{A}\circ\textfrak{G}^{L}$, where $\mathfrak{G}$ be an ideal of approximable operators between arbitrary Banach spaces.

\begin{prop}\label{Auto30}                                          
$(\mathfrak{A})^{L}_{min}$ is a nonlinear ideal. 
\end{prop}

\begin{proof}
The algebraic condition $\bf (\widetilde{PNOI_0})$ is satisfied, since the elementary Lipschitz tensor $g\boxdot e$ admits a factorization $$g\boxdot e : X\stackrel{g\boxdot 1}{\longrightarrow} \mathbb{K}\stackrel{1\otimes 1}{\longrightarrow}\mathbb{K}\stackrel{1\otimes e}{\longrightarrow} F,$$
where $1\otimes e\in\mathfrak{G}\left(\mathbb{K}, F\right)$, $1\otimes 1\in\mathfrak{A}\left(\mathbb{K}, \mathbb{K}\right)$, and $g\boxdot 1\in\textfrak{G}^{L}\left(X, \mathbb{K}\right)$. To prove the algebraic condition $\bf (\widetilde{PNOI_1})$, let $T_{i}\in\mathfrak{G}\circ\mathfrak{A}\circ\textfrak{G}^{L}(X, F)$. Then $T_{i}=B_{i}T_{0}^{i} A_{i}$, where $B_{i}\in\mathfrak{G}(F_{0}^{i}, F)$, $T_{0}^{i}\in\mathfrak{A}(G_{0}^{i}, F_{0}^{i})$, and $A_{i}\in\textfrak{G}^{L}(X, G_{0}^{i})$. Put $B:=B_{1}\circ Q_{1} + B_{2}\circ Q_{2}$, $T_{0}:=\tilde{J}_{1}\circ T_{0}^{1}\circ \tilde{Q}_{1} + \tilde{J}_{2}\circ T_{0}^{2}\circ \tilde{Q}_{2}$, and $A:=J_{1}\circ A_{1} + J_{2}\circ A_{2}$. Now $T_{1} + T_{2}= B\circ T_{0}\circ A$, $B\in\mathfrak{G}(F_{0}, F)$, $T_{0}\in\mathfrak{A}(G_{0}, F_{0})$, and $A\in\textfrak{G}^{L}(X, G_{0})$ imply $T_{1} + T_{2}\in\mathfrak{G}\circ\mathfrak{A}\circ\textfrak{G}^{L}(X, F)$. Let $A\in \Lip(X_{0},X)$, $T\in\mathfrak{G}\circ\mathfrak{A}\circ\textfrak{G}^{L}(X, F)$, and $B\in\mathfrak{L}(F,R_{0})$. Then $T$ admits a factorization 
$$T: X\stackrel{\widetilde{A}}{\longrightarrow} G_{0}\stackrel{T_{0}}{\longrightarrow} F_{0}\stackrel{\widetilde{B}}{\longrightarrow} F,$$
where $\widetilde{B}\in\mathfrak{G}(F_{0}, F)$, $T_{0}\in\mathfrak{A}(G_{0}, F_{0})$, and $\widetilde{A}\in\textfrak{G}^{L}(X, G_{0})$. To show that $BTA\in\mathfrak{G}\circ\mathfrak{A}\circ\textfrak{G}^{L}(X_{0},R_{0})$. By using the linear and nonlinear composition ideal properties, we obtain $B\circ\widetilde{B}\in\mathfrak{G}\left(F_{0}, R_{0}\right)$ and $\widetilde{A}\circ A\in\textfrak{G}^{L}\left(X_{0}, G_{0}\right)$. Hence the Lipschitz operator $BTA$ admits a factorization 
$$BTA: X_{0}\stackrel{\widetilde{\widetilde{A\,}}}{\longrightarrow} G_{0}\stackrel{T_{0}}{\longrightarrow} F_{0}\stackrel{\widetilde{\widetilde{B\,}}}{\longrightarrow} R_{0},$$
where $\widetilde{\widetilde{B\,}}=B\circ\widetilde{B}$ and $\widetilde{\widetilde{A\,}}=\widetilde{A}\circ A$, hence the algebraic condition $\bf (\widetilde{PNOI_2})$ is satisfied.  
\end{proof}

\begin{prop}\label{Auto31}
The rule $$min: \mathfrak{A}\longrightarrow(\mathfrak{A})^{L}_{min}$$
is a monotone Lipschitz procedure. 
\end{prop}

\begin{remark}\label{remmin}

\begin{enumerate}
  \item[$\bf (1)$] It is evident $(\mathfrak{A})^{L}_{min}\subseteq\textfrak{G}^{L}$.
	\item[$\bf (2)$] If $\textfrak{A}^{L}$ is a closed nonlinear ideal, then $(\mathfrak{A})^{L}_{min}\subseteq\textfrak{A}^{L}$.
\end{enumerate}
\end{remark}

\begin{prop}\label{propmin}
 Let  $\mathfrak{A}$ be a closed ideal. Then $(\mathfrak{A})^{L}_{min}=\textfrak{G}^{L}$. 
 In particular, the linear and nonlinear ideals of approximable  operators are related by $(\mathfrak{G})^{L}_{min} = \textfrak{G}^{L}$. 
\end{prop}

The prove of the counterpart of this proposition for ideals of linear operators requires the notion of idempotence of ideals, see \cite[Prop. 4.8.4]{P78}. In particular, the equalities 
\begin{equation}\label{idemFG}
  \mathfrak{F} \circ \mathfrak{F} = \mathfrak{F} \ \ \ \text{and} \ \ \ \mathfrak{G} \circ \mathfrak{G} = \mathfrak{G}
\end{equation}
are needed.
Since idempotence does not make sense for nonlinear ideals, we instead use the following equalities.

\begin{prop}\label{idemFlGl}
  Any Lipschitz finite operator can be written as a product of a linear operator with finite rank and a Lipschitz finite operator. 
	Any Lipschitz approximable operator can be written as a product of a linear approximable operator  and a Lipschitz approximable operator. 
	That is, $\mathfrak{F} \circ \textfrak{F}^L = \textfrak{F}^L$ and $\mathfrak{G} \circ \textfrak{G}^L = \textfrak{G}^L$. 
\end{prop}

\begin{proof}
Let $T=\sum\limits_{j=1}^{m} g_{j}\boxdot e_{j}$ with  $g_{1},\cdots, g_{m}$ in $X^{\#}$ and $e_{1},\cdots, e_{m}$ in $F$ be a Lipschitz finite operator. Let $F_0$ be the finite dimensional subspace of $F$ spanned by $e_{1},\cdots, e_{m}$ and let $J:F_0 \to F$ be the embedding. Obviously, $J$ is a linear operator with finite rank. Moreover, let $T_0$ be the operator $T$ considered as an operator from $X$ to $F_0$. Then $T=J T_0$ is the required factorization. Observe that we also have $\Lip(T_0) \|J\|=\Lip(T)$. The inclusion $\mathfrak{F} \circ \textfrak{F}^L \subseteq \textfrak{F}^L$ is obvious.

Now let $T \in \textfrak{G}^L(X,F)$. Since $T$ can be approximated by Lipschitz finite operators, we can also find Lipschitz finite operators $T_n\in \textfrak{F}^L(X,F)$ such that the sum $T=\sum\limits_{n=1}^{\infty} T_n$ converges absolutely in $\Lip(X,F)$, i.e. $\sum\limits_{n=1}^{\infty} \Lip(T_n) < \infty$. Now each $T_n$ can be factored as $T_n=V_n U_n$ with
$V_n \in \mathfrak{F}(F_i,F)$ and $U_n \in \textfrak{F}^L(X,F_i)$ such that $F_i$ is a suitable Banach space and $\|V_n\| \Lip(U_n) =\Lip(T)$. By homogeneity, we may assume that $\|V_n\|^2 = \Lip(U_n)^2 = \Lip(T_n)$. Let $M:=\ell_{2}(F_{n})$ and put $$V:=\sum\limits_{n=1}^{\infty} J_{n} V_{n}\ \ \  \text{and} \ \ \  U:=\sum\limits_{n=1}^{\infty} U_{n} Q_{n}.$$ Then $$\Lip(U)^{2}\leq\sum\limits_{n=1}^{\infty} \Lip(U_{n})^{2}<\infty\ \ \  \text{and} \ \ \  \left\|V\right\|^{2}\leq\sum\limits_{n=1}^{\infty} \left\|V_{n}\right\|^{2}<\infty.$$ This shows that the series defining $U$ and $V$ are absolutely convergent in $\Lip(X,M)$ and $\mathfrak{L}(M,F)$, respectively. Hence $U \in \textfrak{G}^L(X,M)$ and $V \in \mathfrak{G}(M,F)$ and $T=VU$ is the required factorization. Again, $\mathfrak{G} \circ \textfrak{G}^L \subseteq \textfrak{G}^L$ is obvious. 
\end{proof}

We can now prove Proposition \ref{propmin}.

\begin{proof}[Proof of Proposition \ref{propmin}]
 By \eqref{idemFG} and Proposition \ref{idemFlGl}, we have  
 $$(\mathfrak{G})^{L}_{min} = \mathfrak{G} \circ \mathfrak{G} \circ \textfrak{G}^{L} = \mathfrak{G} \circ \textfrak{G}^{L} = \textfrak{G}^{L}.$$
 If now $\mathfrak{A}$ is a closed ideal, then $\mathfrak{G} \subseteq \mathfrak{A}$ implies
 $$ \textfrak{G}^{L} = (\mathfrak{G})^{L}_{min} = \mathfrak{G} \circ \mathfrak{G} \circ \textfrak{G}^{L} \subseteq \mathfrak{G} \circ \mathfrak{A} \circ \textfrak{G}^{L} = (\mathfrak{A})^{L}_{min}.$$
 The reverse implication was already observed in Remark \ref{remmin}. 
\end{proof}


\subsection{Lipschitz interpolative ideal procedure between metric spaces and Banach spaces}\label{Auto20}

\begin{prop}
$\left[\textfrak{A}^{L}_{\text{inj}}, \mathbf{A}^{L}_{\text{inj}}\right]$ is a $p$-normed nonlinear ideal.  
\end{prop}

\begin{proof}
By Proposition \ref{Auto29} the algebraic conditions of Definition \ref{A7777} are hold. Then the injective hull $\textfrak{A}^{L}_{inj}$  is a nonlinear ideal. To prove the norm condition $\bf (\widetilde{PNOI_1})$, let $T$ and $S$ in $\textfrak{A}^{L}_{\text{inj}}(X,F)$. Then
\begin{align}
\mathbf{A}^{L}_{\text{inj}}\left(S+T\right)^{p}&:=\mathbf{A}^{L}\left[J_{F} \left(S+T\right)\right]^{p}=\mathbf{A}^{L}\left[J_{F} S+ J_{F} T\right]^{p} \nonumber \\
&\leq\mathbf{A}^{L}\left(J_{F} S\right)^{p} + \mathbf{A}^{L}\left(J_{F} T\right)^{p} \nonumber \\
&=\mathbf{A}^{L}_{\text{inj}}(S)^{p} + \mathbf{A}^{L}_{\text{inj}}(T)^{p}. \nonumber
\end{align}
Let $A\in \Lip(X_{0},X)$, $T\in\textfrak{A}^{L}_{\text{inj}}(X,F)$, and $B\in\mathfrak{L}(F,F_{0})$. Then 
\begin{align}
\mathbf{A}^{L}_{\text{inj}}\left(BTA\right)&:=\mathbf{A}^{L}\left(J_{F} \left(BTA\right)\right)=\mathbf{A}^{L}\left(B^{\text{inj}} \left(J_{F} T\right) A\right)                               \nonumber \\
&\leq\left\|B^{\text{inj}}\right\|\mathbf{A}^{L}\left(J_{F} T\right) \Lip(A)                      \nonumber \\
&=\left\|B\right\|\mathbf{A}^{L}_{\text{inj}}(T) \Lip(A).                                        \nonumber 
\end{align}
Hence the norm condition $\bf (\widetilde{PNOI_2})$ is satisfied.  
\end{proof}

\begin{definition}\label{Torezen}
Let $0\leq \theta< 1$ and $\left[\textfrak{A}^{L}, \mathbf{A}^{L}\right]$ be a normed nonlinear ideal. A Lipschitz operator $T$ from $X$ into $F$ belongs to $\textfrak{A}^{L}_{\theta}(X, F)$ if there exist a Banach space $G$ and a Lipschitz operator $S\in\textfrak{A}^{L}(X, G)$ such that 
\begin{equation}\label{pooor}
\left\|Tx'-Tx''|F\right\|\leq\left\|Sx'-Sx''|G\right\|^{1-\theta}\cdot d_{X}(x',x'')^{\theta},\ \  \forall\: x',\; x'' \in X.
\end{equation}
For each $T\in\textfrak{A}^{L}_{\theta}(X, F)$, we set 

\begin{equation}\label{pooooor}
\mathbf{A}^{L}_{\theta}(T):=\inf\mathbf{A}^{L}(S)^{1-\theta}
\end{equation}
where the infimum is taken over all Lipschitz operators $S$ admitted in (\ref{pooor}). 

Note that $\Lip(T)\leq\mathbf{A}^{L}_{\theta}(T)$, by definition.
\end{definition}

\begin{prop}\label{Flughafen}
$\left[\textfrak{A}^{L}_{\theta}, \mathbf{A}^{L}_{\theta}\right]$ is an injective Banach nonlinear ideal. 
\end{prop}

\begin{proof}
The condition of injective hull  and the algebraic conditions of Definition \ref{A7777} are not difficult to prove it. Let $x',\; x'' \in X$ and $e\in F$ the norm condition $\bf (\widetilde{PNOI_0})$ is satisfied. Indeed
\begin{equation}
\left\|(g\boxdot e) x'- (g\boxdot e) x''|F\right\|\leq\left\|\left(\Lip(g)\cdot\left\|e\right\|\right)^{\frac{\theta}{1-\theta}}\left[(g\boxdot e) x'- (g\boxdot e) x''\right]|F\right\|^{1-\theta}\cdot d_{X}(x',x'')^{\theta}.
\end{equation}
Since a Lipschitz operator $S:=(\Lip(g)\cdot\left\|e\right\|)^{\frac{\theta}{1-\theta}} (g\boxdot e)\in\textfrak{A}^{L}(X, F)$, hence $g\boxdot e\in\textfrak{A}^{L}_{\theta}(X, F)$. From (\ref{pooooor}) we have 
\begin{align}
\mathbf{A}^{L}_{\theta}(g\boxdot e)&:=\inf\mathbf{A}^{L}\big((\Lip(g)\cdot\left\|e\right\|)^{\frac{\theta}{1-\theta}} g\boxdot e\big)^{1-\theta} \nonumber \\
&\leq(\Lip(g)\cdot\left\|e\right\|)^{\theta}\mathbf{A}^{L}(g\boxdot e\big)^{1-\theta} \nonumber \\
&=(\Lip(g)\cdot\left\|e\right\|)^{\theta}\cdot(\Lip(g)\cdot\left\|e\right\|)^{1-\theta} \nonumber \\
&=\Lip(g)\cdot\left\|e\right\|.\nonumber 
\end{align}
The converse inequality is obvious. To prove the norm condition $\bf (\widetilde{PNOI_1})$, let $T_{1}$ and $T_{2}$ in $\textfrak{A}^{L}_{\theta}(X, F)$. Given $\epsilon > 0$, there is a Banach space $G_{i}$ and a Lipschitz operator $S_{i}\in\textfrak{A}^{L}(X, G_{i})$, $i=1, 2$ such that  
\begin{equation}
\left\|T_{i}x'-T_{i}x''|F\right\|\leq\left\|\mathbf{A}^{L}(S_{i})^{-\theta}\left(S_{i}x'-S_{i}x''\right)|G\right\|^{1-\theta} \left(\mathbf{A}^{L}(S_{i})^{1-\theta}\right)^{\theta}\cdot d_{X}(x',x'')^{\theta},\ \  \forall\: x',\; x'' \in X
\end{equation}
and $\mathbf{A}^{L}(S_{i})^{1-\theta}\leq (1+\epsilon)\cdot \mathbf{A}^{L}_{\theta}(T_{i})\ \left(i=1, 2\right)$. Introducing the $\ell_{1}$-sum $G:=G_{1}\oplus G_{2}$ and the Lipschitz operator $S:=\mathbf{A}^{L}(S_{1})^{-\theta} J_{1} S_{1} + \mathbf{A}^{L}(S_{2})^{-\theta} J_{2} S_{2}\in\textfrak{A}^{L}(X, G)\  \left(J_{1}, J_{2}\  \text{the canonical injections}\right)$ and applying the H\"older inequality, we get for all $x', x''\in X$
\begin{align}
\left\|(T_{1}+T_{2})x' - (T_{1}+T_{2})x''|F\right\|&\leq\left\|T_{1}x'+ T_{1}x''|F\right\|+\left\|T_{1}x'+ T_{1}x''|F\right\|\nonumber \\
&\leq\sum\limits_{i=1}^{2}\left\|\mathbf{A}^{L}(S_{i})^{-\theta}\left(S_{i}x'-S_{i}x''\right)|G_{i}\right\|^{1-\theta}\left(\mathbf{A}^{L}(S_{i})^{1-\theta}\right)^{\theta}\cdot d_{X}(x',x'')^{\theta} \nonumber \\
&\leq\left[\sum\limits_{i=1}^{2}\left\|\mathbf{A}^{L}(S_{i})^{-\theta}\left(S_{i}x'-S_{i}x''\right)|G_{i}\right\|\right]^{1-\theta}\left(\sum\limits_{i=1}^{2}\mathbf{A}^{L}(S_{i})^{1-\theta}\right)^{\theta} d_{X}(x',x'')^{\theta} \nonumber \\
&=\left(\mathbf{A}^{L}(S_{1})^{1-\theta} + \mathbf{A}^{L}(S_{2})^{1-\theta}\right)^{\theta}\left\|Sx'-Sx''|G\right\|^{1-\theta} \cdot d_{X}(x',x'')^{\theta}. \nonumber
\end{align}
Hence $T_{1}+ T_{2}\in\textfrak{A}^{L}_{\theta}(X, F)$ and furthermore, for $p=1$ we have
\begin{align}
\mathbf{A}^{L}_{\theta}(T_{1} + T_{2})&\leq\left[\mathbf{A}^{L}(S_{1})^{1-\theta} + \mathbf{A}^{L}(S_{2})^{1-\theta}\right]^{\theta}\mathbf{A}^{L}(S)^{1-\theta}\nonumber \\
&\leq\mathbf{A}^{L}(S_{1})^{1-\theta} + \mathbf{A}^{L}(S_{2})^{1-\theta} \nonumber \\
&\leq (1+\epsilon)\cdot\left(\mathbf{A}^{L}_{\theta}(T_{1})+\mathbf{A}^{L}_{\theta}(T_{2})\right). \nonumber
\end{align}

To prove the norm condition $\bf (\widetilde{PNOI_2})$, let $A\in \mathscr{L}(X_{0},X)$, $T\in\textfrak{A}^{L}_{\theta}(X, F)$, and $B\in\mathfrak{L}(F, F_{0})$ and $x'_{0}, x''_{0}$ in $X$. Then
\begin{align}
\left\|BTA x'_{0} - BTA x''_{0}|F_{0}\right\|&\leq \left\|B\right\|\cdot \left\|TA x'_{0} - TA x''_{0}|F\right\| \nonumber \\
&\leq\left\|B\right\|\cdot\left\|S (Ax'_{0})- S (Ax''_{0})| G\right\|^{1-\theta}\cdot d_{X}(Ax'_{0},Ax''_{0})^{\theta} \nonumber \\
&\leq \left\|B\right\|\cdot\Lip(A)^{\theta}\cdot \left\|S\circ A (x'_{0}) - S\circ A (x''_{0})| G\right\|^{1-\theta} d_{X_{0}}(x'_{0},x''_{0})^{\theta} \nonumber \\
&\leq \left\|\left\|B\right\|^{\frac{1}{1-\theta}}\cdot\Lip(A)^{\frac{\theta}{1-\theta}}   \left(S\circ A (x'_{0})-S\circ A (x''_{0})\right)| G\right\|^{1-\theta} d_{X_{0}}(x'_{0},x''_{0})^{\theta}. 
\end{align}
Since a Lipschitz map $\widetilde{S}:=\left(\left\|B\right\|^{\frac{1}{1-\theta}}\cdot\Lip(A)^{\frac{\theta}{1-\theta}}\right) S\circ A\in\textfrak{A}^{L}(X_{0}, G)$, hence $BTA\in\textfrak{A}^{L}_{\theta}(X_{0}, F_{0})$. Moreover, from (\ref{pooooor}) we have
\begin{align}\label{joo}
\mathbf{A}^{L}_{\theta}(BTA)&:=\inf \mathbf{A}^{L}(\widetilde{S})^{1-\theta} \nonumber \\
&\leq\mathbf{A}^{L}\left((\left\|B\right\|^{\frac{1}{1-\theta}}\cdot\Lip(A)^{\frac{\theta}{1-\theta}}) S\circ A\right)^{1-\theta} \nonumber \\
&\leq \left\|B\right\|\cdot\Lip(A)^{\theta}\cdot \mathbf{A}^{L}(S\circ A)^{1-\theta} \nonumber \\
&= \left\|B\right\|\cdot\Lip(A)\cdot\mathbf{A}^{L}(S)^{1-\theta}. 
\end{align}
Taking the infimum over all such $S\in\textfrak{A}^{L}(X, G)$ on the right side of (\ref{joo}), we have
$$\mathbf{A}^{L}_{\theta}(BTA)\leq\left\|B\right\|\cdot\mathbf{A}^{L}_{\theta}(T)\cdot\Lip(A).$$

To prove the completeness, condition $\bf (\widetilde{PNOI_3})$, let $(T_{n})_{n\in\mathbb{N}}$ be a sequence of Lipschitz operator in $\textfrak{A}^{L}_{\theta}(X, F)$ such that $\sum\limits_{n=1}^{\infty}\mathbf{A}^{L}_{\theta}(T_{n})<\infty$. Since $\Lip(T)\leq\mathbf{A}^{L}_{\theta}(T)$ and $\Lip(X, F)$ is a Banach space, there exists $T=\sum\limits_{n=1}^{\infty} T_{n}\in\Lip(X, F)$. Let $S_{n}\in\textfrak{A}^{L}(X, G_{n})$ such that $$\left\|T_{n}x'-T_{n}x''|F\right\|\leq\left\|S_{n}x'-S_{n}x''|G_{n}\right\|^{1-\theta}\cdot d_{X}(x',x'')^{\theta},\ \  \forall\: x',\; x'' \in X,$$
and $\mathbf{A}^{L}(S_{n})^{1-\theta}\leq\mathbf{A}^{L}_{\theta}(T_{n})+\frac{\epsilon}{2^{n}}$. Then 
$$\left(\sum\limits_{n=1}^{\infty}\mathbf{A}^{L}(S_{n})\right)^{1-\theta}\leq\sum\limits_{n=1}^{\infty}\mathbf{A}^{L}(S_{n})^{1-\theta}\leq\sum\limits_{n=1}^{\infty}\mathbf{A}^{L}_{\theta}(T_{n})+\epsilon<\infty.$$
Let $S=\sum\limits_{n=1}^{\infty} S_{n}\in\textfrak{A}^{L}(X, G)$, where $G$ is the $\ell_{1}$-sum of all $G_{n}$. Hence 
\begin{align}
\left\|Tx'-Tx''|F\right\|\leq\sum\limits_{n=1}^{\infty}\left\|T_{n}x'-T_{n}x''|F_{n}\right\|&\leq\sum\limits_{n=1}^{\infty}\left\|S_{n}x'-S_{n}x''|G_{n}\right\|^{1-\theta}\cdot d_{X}(x',x'')^{\theta}\nonumber \\
&\leq\left\|Sx'-Sx''|G_{n}\right\|^{1-\theta}\left(\sum\limits_{n=1}^{\infty}\mathbf{A}^{L}(S_{n})^{1-\theta}\right)^{\theta}\cdot d_{X}(x',x'')^{\theta}.\nonumber
\end{align}
This implies that $T\in\textfrak{A}^{L}_{\theta}(X, F)$ and $$\mathbf{A}^{L}_{\theta}(T)\leq\sum\limits_{n=1}^{\infty}\mathbf{A}^{L}(S_{n})^{1-\theta}\leq\sum\limits_{n=1}^{\infty}\mathbf{A}^{L}_{\theta}(T_{n})+\epsilon<\infty.$$ 
We have $$\mathbf{A}^{L}_{\theta}\left(T-\sum\limits_{j=1}^{n} T_{j}\right)=\mathbf{A}^{L}_{\theta}\left(\sum\limits_{k=n+1}^{\infty} T_{k}\right)\leq\sum\limits_{k=n+1}^{\infty}\mathbf{A}^{L}_{\theta}(S_{k})^{1-\theta}.$$
Thus, $T=\sum\limits_{n=1}^{\infty} T_{n}$.
\end{proof}
\begin{remark}
If $\theta=0$, then the nonlinear ideal $\textfrak{A}^{L}_{\theta}$ is just the injective hull of nonlinear ideal $\textfrak{A}^{L}$ and Lipschitz norms are equal. Further properties are given in.
\end{remark}

\begin{prop}
Let $0\leq \theta, \theta_{1}, \theta_{2} < 1$. Then the following holds.
\begin{enumerate}
	\item[$\bf (a)$] $\textfrak{A}^{L}_{\theta_{1}}\subset \textfrak{A}^{L}_{\theta_{2}}$ if $\theta_{1}\leq\theta_{2}$.
	\item[$\bf (b)$] $\textfrak{A}^{L}_{\text{inj}}\subset \textfrak{A}^{L}_{\theta}$.
	\item[$\bf (c)$] $\left(\textfrak{A}^{L}_{\theta_{1}}\right)_{\theta_{2}}\subset\textfrak{A}^{L}_{\theta_{1}+\theta_{2}-\theta_{1}\theta_{2}}$.
\end{enumerate}
\end{prop}

\begin{proof}
To verify $\bf (a)$, let $T\in\textfrak{A}^{L}_{\theta_{1}}(X, F)$ and $\epsilon > 0$. Then
$$\left\|Tx'-Tx''|F\right\|\leq\left\|Sx'-Sx''|G\right\|^{1-\theta_{1}}\cdot d_{X}(x',x'')^{\theta_{1}},\ \  \forall\: x',\; x'' \in X,$$
holds for a suitable Banach space $G$ and a Lipschitz operator $S\in\textfrak{A}^{L}(X, G)$ with $\mathbf{A}^{L}(S)^{1-\theta_{1}}\leq (1+\epsilon)\cdot\mathbf{A}^{L}_{\theta_{1}}(T)$. Since
$$\left\|Tx'-Tx''|F\right\|\leq\Lip(S)^{\theta_{2}-\theta_{1}}\left\|Sx'-Sx''|G\right\|^{1-\theta_{2}}\cdot d_{X}(x',x'')^{\theta_{2}},\ \  \forall\: x',\; x'' \in X,$$
Since a Lipschitz map $\widetilde{S}:=\Lip(S)^{\frac{\theta_{2}-\theta_{1}}{1-\theta_{2}}} S\in\textfrak{A}^{L}(X, G)$, hence $T\in\textfrak{A}^{L}_{\theta_{2}}(X, F)$ and $$\mathbf{A}^{L}_{\theta_{2}}(T)\leq\mathbf{A}^{L}(\widetilde{S})^{1-\theta_{2}}\leq\Lip(S)^{\theta_{2}-\theta_{1}}\mathbf{A}^{L}(S)^{1-\theta_{2}}\leq\mathbf{A}^{L}(S)^{1-\theta_{1}}\leq (1+\epsilon)\cdot\mathbf{A}^{L}_{\theta_{1}}(T).$$
To verify $\bf (b)$, let $T\in\textfrak{A}^{L}_{\text{inj}}(X, F)$. Then $J_{F}T\in\textfrak{A}^{L}(X, F^{\text{inj}})$ and
\begin{align}
\left\|Tx'-Tx''|F\right\|&=\left\|J_{F}(Tx')-J_{F}(Tx'')|F^{\text{inj}}\right\|\nonumber \\
&\leq\Lip(T)^{\theta}\cdot\left\|J_{F}\circ T(x')-J_{F}\circ T(x'')|F^{\text{inj}}\right\|^{1-\theta}\cdot d_{X}(x',x'')^{\theta}. \nonumber 
\end{align}
Since $G:=F^{\text{inj}}$ and a Lipschitz map $S:=\Lip(T)^{\frac{\theta}{1-\theta}} J_{F}\circ T\in\textfrak{A}^{L}(X, G)$, hence $T\in\textfrak{A}^{L}_{\theta}(X, F)$. Moreover,
\begin{align}
\mathbf{A}^{L}_{\theta}(T):=\inf\mathbf{A}^{L}(S)^{1-\theta}&\leq\mathbf{A}^{L}(\Lip(T)^{\frac{\theta}{1-\theta}} J_{F}\circ T)^{1-\theta}  \nonumber \\
&\leq\Lip(T)^{\theta}\cdot\mathbf{A}^{L}(J_{F}\circ T)^{1-\theta}  \nonumber \\
&:=\Lip(T)^{\theta}\cdot\mathbf{A}_{\text{inj}}^{L}(T)^{1-\theta}  \nonumber  \\
&\leq\mathbf{A}_{\text{inj}}^{L}(T)^{\theta}\cdot\mathbf{A}_{\text{inj}}^{L}(T)^{1-\theta}  \nonumber  \\
&=\mathbf{A}_{\text{inj}}^{L}(T).  \nonumber 
\end{align} 

To verify $\bf (c)$, let $T\in\left(\textfrak{A}^{L}_{\theta_{1}}\right)_{\theta_{2}}(X, F)$ and $\epsilon > 0$. Then 
\begin{equation}\label{traurig}
\left\|Tx'-Tx''|F\right\|\leq\left\|Sx'-Sx''|G\right\|^{1-\theta_{2}}\cdot d_{X}(x',x'')^{\theta_{2}},\ \  \forall\: x',\; x'' \in X,
\end{equation}
holds for a suitable Banach space $G$ and a Lipschitz operator $S\in\textfrak{A}_{\theta_{1}}^{L}(X, G)$ with $\mathbf{A}_{\theta_{1}}^{L}(S)^{1-\theta_{2}}\leq (1+\epsilon)\cdot\left(\mathbf{A}_{\theta_{1}}^{L}(T)\right)_{\theta_{2}}$ and 
\begin{equation}\label{traurig1}
\left\|Sx'-Sx''|G\right\|\leq\left\|Rx'-Rx''|G\right\|^{1-\theta_{1}}\cdot d_{X}(x',x'')^{\theta_{1}},\ \  \forall\: x',\; x'' \in X,
\end{equation}
holds for a suitable Banach space $\widetilde{G}$ and a Lipschitz operator $R\in\textfrak{A}^{L}(X, \widetilde{G})$ with $\mathbf{A}^{L}(R)^{1-\theta_{1}}\leq (1+\epsilon)\cdot\mathbf{A}_{\theta_{1}}^{L}(S)$. From (\ref{traurig}) and (\ref{traurig1}) we have
\begin{align}
\left\|Tx'-Tx''|F\right\|&\leq\left\|Rx'-Rx''|\widetilde{G}\right\|^{({1-\theta_{1}})\cdot ({1-\theta_{2}})}\cdot d_{X}(x',x'')^{\theta_{1}\cdot (1-\theta_{2})}\cdot d_{X}(x',x'')^{\theta_{2}}\nonumber \\
&\leq\left\|Rx'-Rx''|\widetilde{G}\right\|^{1-\theta_{2}-\theta_{1}+\theta_{1}\cdot\theta_{2}}\cdot d_{X}(x',x'')^{\theta_{2}+\theta_{1}-\theta_{1}\cdot\theta_{2}},\nonumber  
\end{align}
hence $T\in\textfrak{A}^{L}_{\theta_{1}+\theta_{2}-\theta_{1}\theta_{2}}(X, F)$. Moreover,
\begin{align}
\mathbf{A}^{L}_{\theta_{1}+\theta_{2}-\theta_{1}\theta_{2}}(T)&\leq\mathbf{A}^{L}(R)^{1-\theta_{1}-\theta_{2}+\theta_{1}\theta_{2}}  \nonumber \\
&\leq (1+\epsilon)^{\frac{1-\theta_{1}-\theta_{2}+\theta_{1}\theta_{2}}{1-\theta_{1}}}\cdot\mathbf{A}_{\theta_{1}}^{L}(S)^{\frac{1-\theta_{1}-\theta_{2}+\theta_{1}\theta_{2}}{1-\theta_{1}}}  \nonumber \\
&\leq (1+\epsilon)^{\frac{1-\theta_{1}-\theta_{2}+\theta_{1}\theta_{2}}{1-\theta_{1}}}\cdot (1+\epsilon)^{\frac{1}{1-\theta_{2}}\cdot\frac{1-\theta_{1}-\theta_{2}+\theta_{1}\theta_{2}}{1-\theta_{1}}}\cdot\left(\mathbf{A}_{\theta_{1}}^{L}(T)\right)_{\theta_{2}}^{\frac{1}{1-\theta_{2}}\cdot\frac{1-\theta_{1}-\theta_{2}+\theta_{1}\theta_{2}}{1-\theta_{1}}}  \nonumber \\
&= (1+\epsilon)^{\frac{2-2\theta_{1}-\theta_{2}+\theta_{1}\theta_{2}}{1-\theta_{1}}}\cdot\left(\mathbf{A}_{\theta_{1}}^{L}(T)\right)_{\theta_{2}}.  \nonumber 
\end{align} 
\end{proof}


\subsection{Lipschitz $\left(p,\theta, q, \nu\right)$-dominated operators}\label{Auto220}

$1\leq p, q <\infty$ and $0\leq \theta, \nu< 1$ such that $\frac{1}{r}+\frac{1-\theta}{p}+\frac{1-\nu}{q}=1$ with $1\leq r <\infty$. We then introduce the following definition.

\begin{definition}\label{ten1}
  A Lipschitz operator $T$ from $X$ to $F$ is called Lipschitz $\left(p,\theta, q, \nu\right)$-dominated if there exists a Banach spaces $G$ and $H$, a Lipschitz operator $S\in\Pi_{p}^{L}(X,G)$, a bounded operator $R\in\Pi_{q}(F^{*},H)$ and a positive constant $C$ such that

\begin{equation}\label{lastwagen1}
\left|\left\langle Tx'- Tx'', b^{*}\right\rangle\right|\leq C\cdot d_{X}(x', x'')^{\theta}\left\|Sx'-Sx''| G\right\|^{1-\theta}\left\|b^{*}\right\|^{\nu}\left\|R(b^{*})|H\right\|^{1-\nu}
\end{equation}
for arbitrary finite sequences $x'$, $x''$ in $X$, and $b^{*}\subset F^{*}$. 

Let us denote by $\mathcal{D}_{\left(p, \theta, q, \nu\right)}^{L}(X,F)$ the class of all Lipschitz $\left(p, \theta, q, \nu\right)$-dominated operators from $X$ to $F$ with $$D_{\left(p, \theta, q, \nu\right)}^{L}(T)=\inf\left\{C \cdot\pi_{p}^{L}(S)^{1-\theta}\cdot \pi_{q}(R)^{1-\nu}\right\},$$
where the infimum is taken over all Lipschitz operator $S$, bounded operator $R$, and constant $C$ admitted in (\ref{lastwagen1}). 
\end{definition}

\begin{prop}\label{sun1}
The ordered pair $\left(\mathcal{D}_{\left(p, \theta, q, \nu\right)}^{L}(X,F), D_{\left(p, \theta, q, \nu\right)}^{L}\right)$ is a normed space.
\end{prop}

\begin{proof} 
We prove the triangle inequality. Let $i=1,2$ and $T_{i}\in\mathcal{D}_{\left(p, \theta, q, \nu\right)}^{L}(X,F)$. For each $\epsilon> 0$, there exists a Banach spaces $G_{i}$ and $H_{i}$, a Lipschitz operators $S_{i}\in\Pi_{p}^{L}(X,G_{i})$, a bounded operators $R_{i}\in\Pi_{q}(F^{*},H_{i})$ and a positive constants $C_{i}$
such that 

\begin{equation}\label{lastwagen2}
\left|\left\langle T_{i}x'- T_{i}x'', b^{*}\right\rangle\right|\leq C_{i} d_{X}(x', x'')^{\theta}\left\|S_{i}x'-S_{i}x''| G_{i}\right\|^{1-\theta}\left\|b^{*}\right\|^{\nu}\left\|R_{i}(b^{*})|H_{i}\right\|^{1-\nu} \forall x', x'' \in X \; \forall b^{*}\subset F^{*}
\end{equation}

and

\begin{equation}\label{lastwagen3}
C_{i} \cdot\pi_{p}^{L}(S_{i})^{1-\theta}\cdot \pi_{q}(R_{i})^{1-\nu}\leq D_{\left(p, \theta, q, \nu\right)}^{L}(T_{i}) + \epsilon.
\end{equation}

For $x', x'' \in X$ and $b^{*}$ we have

\begin{align}\label{lastwagen4}
\left|\left\langle T_{i}x'- T_{i}x'', b^{*}\right\rangle\right|&\leq C_{i} \cdot d_{X}(x', x'')^{\theta}\left\|S_{i}x'-S_{i}x''| G_{i}\right\|^{1-\theta}\left\|b^{*}\right\|^{\nu}\left\|R_{i}(b^{*})|H_{i}\right\|^{1-\nu} \nonumber \\
&=\widetilde{C}_{i} \cdot d_{X}(x', x'')^{\theta}\left\|\widetilde{S}_{i}x'-\tilde{S}_{i}x''| G_{i}\right\|^{1-\theta}\left\|b^{*}\right\|^{\nu}\left\|\widetilde{R}_{i}(b^{*})|H_{i}\right\|^{1-\nu}, \nonumber 
\end{align}

where $\widetilde{C}_{i}=C_{i}^{\frac{1}{r}}\cdot\pi_{p}^{L}(S_{i})^{\frac{1-\theta}{r}}\cdot\pi_{q}(R_{i})^{\frac{1-\nu}{r}}$, $\widetilde{S}_{i}=C_{i}^{\frac{1}{p}}\cdot\pi_{p}^{L}(S_{i})^{\frac{1-\theta}{p}}\cdot\pi_{q}(R_{i})^{\frac{1-\nu}{p}}\frac{S_{i}}{\pi_{p}^{L}(S_{i})}$, and $\widetilde{R}_{i}=C_{i}^{\frac{1}{q}}\cdot\pi_{q}^{L}(S_{i})^{\frac{1-\theta}{q}}\cdot\pi_{q}(R_{i})^{\frac{1-\nu}{p}}\frac{R_{i}}{\pi_{q}^{L}(R_{i})}$.

From (\ref{lastwagen2}) and (\ref{lastwagen3}) we have $$C_{i}\leq \left(D_{\left(p, \theta, q, \nu\right)}^{L}(T_{i}) + \epsilon\right)^{\frac{1}{r}}.$$

\begin{equation}\label{lastwagen5}
\pi_{p}^{L}(S_{i})\leq \left(D_{\left(p, \theta, q, \nu\right)}^{L}(T_{i}) + \epsilon\right)^{\frac{1}{p}} \ \text{and}\  \pi_{q}(R_{i})\leq \left(D_{\left(p, \theta, q, \nu\right)}^{L}(T_{i}) + \epsilon\right)^{\frac{1}{q}}.
\end{equation}

Let $G$ and $H$ be a Banach spaces obtained as a direct sum of $\ell_{p}$ and $\ell_{q}$ by $G_{1}$ and $G_{2}$ and $H_{1}$ and $H_{2}$ respectively. Let $S$ be a Lipschitz operator from $X$ into $G$ such that $S(x)=(S_{i}(x))_{i=1}^{2}$ for $x\in X$ and $R$ be a bounded operator from $F^{*}$ into $H$ such that $R(b)=(R_{i}(b))_{i=1}^{2}$ for $b\in F^{*}$. For each finite sequence $x'$, $x''$ in $X$ we have

\begin{align}
\left\|(S(x'_{j})-S(x''_{j}))_{j=1}^{n}|\ell_{p}(G)\right\|&=\left[\sum\limits_{j=1}^{n}\left\|S(x'_{j})-S(x''_{j})|G\right\|^{p}\right]^{\frac{1}{p}}=\left[\sum\limits_{j=1}^{n}\sum\limits_{i=1}^{2}\left\|S_{i}(x'_{j})-S_{i}(x''_{j})|G\right\|^{p}\right]^{\frac{1}{p}} \nonumber \\
&\leq\left[\sum\limits_{i=1}^{2}\pi_{p}^{L}(S_{i})^{p}\sup\limits_{f\in B_{{X}^{\#}}}\sum\limits_{j=1}^{n}\left|fx'_j-fx''_j\right|^{p}\right]^{\frac{1}{p}} \nonumber \\
&=\sup\limits_{f\in B_{{X}^{\#}}}\left[\sum\limits_{j=1}^{n}\left|fx'_j-fx''_j\right|^{p}\right]^{\frac{1}{p}}\left(\sum\limits_{i=1}^{2}\pi_{p}^{L}(S_{i})^{p}\right)^{\frac{1}{p}} \nonumber
\end{align}

\begin{equation}
\pi_{p}^{L}(S)\leq\left(\sum\limits_{i=1}^{2}\pi_{p}^{L}(S_{i})^{p}\right)^{\frac{1}{p}}\leq\left(D_{\left(p, \theta, q, \nu\right)}^{L}(T_{1})+D_{\left(p, \theta, q, \nu\right)}^{L}(T_{2}) + 2\epsilon\right)^{\frac{1}{p}}.
\end{equation}

\begin{equation}
\pi_{q}(B)\leq\left(D_{\left(p, \theta, q, \nu\right)}^{L}(T_{1})+D_{\left(p, \theta, q, \nu\right)}^{L}(T_{2}) + 2\epsilon\right)^{\frac{1}{q}}.
\end{equation}

\begin{align}
&\left|\left\langle (T_{1}+T_{1})x'- (T_{1}+T_{1}) x'', b^{*}\right\rangle\right|\leq\sum\limits_{i=1}^{2} C_{i} d_{X}(x', x'')^{\theta}\left\|S_{i}x'-S_{i}x''| G_{i}\right\|^{1-\theta}\left\|b^{*}\right\|^{\nu}\left\|R_{i}(b^{*})|H_{i}\right\|^{1-\nu} \nonumber \\
&\leq d_{X}(x', x'')^{\theta}\left\|b^{*}\right\|^{\nu}\left(\sum\limits_{i=1}^{2} C_{i}^{r}\right)^{\frac{1}{r}} \left(\sum\limits_{i=1}^{2} \left\|S_{i}x'-S_{i}x''| G_{i}\right\|^{p}\right)^{\frac{1-\theta}{p}} \left(\sum\limits_{i=1}^{2} \left\|R_{i}(b^{*})|H_{i}\right\|^{q}\right)^{\frac{1-\nu}{q}}             \nonumber \\
&=d_{X}(x', x'')^{\theta}\left\|b^{*}\right\|^{\nu}\left(\sum\limits_{i=1}^{2} C_{i}^{r}\right)^{\frac{1}{r}}\left\|Sx'-Sx''| G\right\|^{1-\theta} \left\|R(b^{*})|H\right\|^{1-\nu} \nonumber
\end{align}

\begin{align}
D_{\left(p, \theta, q, \nu\right)}^{L}(T_{1}+T_{2})&\leq\left(\sum\limits_{i=1}^{2} C_{i}^{r}\right)^{\frac{1}{r}}\pi_{p}^{L}(S)^{1-\theta}\pi_{q}(R)^{1-\nu} \nonumber \\
&\leq \left(D_{\left(p, \theta, q, \nu\right)}^{L}(T_{1})+ D_{\left(p, \theta, q, \nu\right)}^{L}(T_{2})+ 2\epsilon\right)^{\frac{1}{r}+\frac{1-\theta}{p}+\frac{1-\nu}{q}}.           \nonumber 
\end{align}

Hence $D_{\left(p, \theta, q, \nu\right)}^{L}(T_{1}+T_{2})\leq D_{\left(p, \theta, q, \nu\right)}^{L}(T_{1})+ D_{\left(p, \theta, q, \nu\right)}^{L}(T_{2})$

\end{proof}

\begin{remark} 
 If $\theta=\nu=0$, then the class of all Lipschitz $\left(p, \theta, q, \nu\right)$-dominated operators from $X$ to $F$ are the class of all Lipschitz $\left(p, q\right)$-dominated operators from $X$ to $F$ considered in \cite{CD11} with $\mathcal{D}_{\left(p, \theta, q, \nu\right)}^{L}(X,F)=\mathcal{D}_{\left(p, q\right)}^{L}(X,F)$.
\end{remark}

\begin{thm}\label{thm1}
Let $X$ be a metric space, $F$ be a Banach space and $T\in Lip(X,F)$. The following conditions are equivalent.       

\begin{enumerate}
	\item[$\bf (1)$] $T\in\mathcal{D}_{\left(p, \theta, q, \nu\right)}^{L}(X,F)$.
	\item[$\bf (2)$] There is a constant $C\geq 0$ and regular probabilities $\mu$ and $\tau$ on $B_{X^{\#}}$ and $B_{F^{**}}$, respectively such that for every $x'$, $x''$ in $X$ and $b^{*}$ in $F^{*}$ the following inequality holds
\begin{align}
\left|\left\langle Tx'- Tx'', b^{*}\right\rangle\right|\leq C &\cdot\left[\int\limits_{B_{X^{\#}}}\left(\left|f(x')-f(x'')\right|^{1-\theta} d_X(x',x'')^{\theta}\right)^\frac{p}{1-\theta} d\mu(f)\right]^\frac{1-\theta}{p}\nonumber \\
&\cdot\left[\int\limits_{B_{F^{**}}}\left(\left|\left\langle b^{*}, b^{**}\right\rangle\right|^{1-\nu} \left\|b^{*}\right\|^{\nu}\right)^\frac{q}{1-\nu} d\tau(b^{**})\right]^\frac{1-\nu}{q}. \nonumber
\end{align}
	\item[$\bf (3)$]  There exists a constant $C\geq 0$ such that for every finite sequences $x'$, $x''$ in $X$; $\sigma$ in $\mathbb{R}$ and $y^{*}\subset F^{*}$ the inequality
\begin{equation}
\left\|\sigma\cdot\left\langle Tx'- Tx'', b^{*}\right\rangle|\ell_{r'}\right\|\leq C\cdot\left\|(\sigma,x',x'')\Big|\delta_{p,\theta}^{L}(\mathbb{R}\times X\times X)\right\|\left\| b^{*} \Big|\delta_{q,\nu}(F^{*})\right\|
\end{equation}
holds. 
\end{enumerate}
In this case, $D_{\left(p, \theta, q, \nu\right)}^{L}$ is equal to the infimum of such constants $C$ in either $\bf (2)$, or $\bf (3)$.  
\end{thm}

\begin{proof} 
$\bf (1)\Longrightarrow\bf (2)$ If $T\in\mathcal{D}_{\left(p, \theta, q, \nu\right)}^{L}(X,F)$, then there exists a Banach spaces $G$ and $H$, a Lipschitz operator $S_{1}\in\Pi_{p}^{L}(X,G)$, a bounded operator $S_{2}\in\Pi_{q}(F^{*},H)$ and a positive constant $C$ such that 
\begin{equation}
\left|\left\langle Tx'- Tx'', b^{*}\right\rangle\right|\leq C\cdot d_{X}(x', x'')^{\theta}\left\|S_{1}x'-S_{1}x''| G\right\|^{1-\theta}\left\|b^{*}\right\|^{\nu}\left\|S_{2}(b^{*})|H\right\|^{1-\nu}
\end{equation}
for arbitrary finite sequences $x'$, $x''$ in $X$ and $b^{*}\subset F^{*}$. Since $S_{1}$ is Lipschitz $p$--summing operator and $S_{2}$ is $q$--summing operator then there exists regular probabilities $\mu$ and $\tau$ on $B_{X^{\#}}$ and $B_{F^{**}}$, respectively such that
\begin{align}
\left|\left\langle Tx'- Tx'', b^{*}\right\rangle\right|\leq C\cdot\pi_{p}^{L}(S_{1})^{1-\theta}\cdot\pi_{q}(S_{2})^{1-\nu}&\left[\int\limits_{B_{X^{\#}}}\left(\left|f(x')-f(x'')\right|^{1-\theta} d_X(x',x'')^{\theta}\right)^\frac{p}{1-\theta} d\mu(f)\right]^\frac{1-\theta}{p}\nonumber \\
&\cdot\left[\int\limits_{B_{F^{**}}}\left(\left|\left\langle b^{*}, b^{**}\right\rangle\right|^{1-\nu} \left\|b^{*}\right\|^{\nu}\right)^\frac{q}{1-\nu} d\tau(b^{**})\right]^\frac{1-\nu}{q}. \nonumber
\end{align}
$\bf (2)\Longrightarrow\bf (1)$ Let $x'$, $x''$ be finite sequences in $X$ and $y^{*}$ be a finite sequence in $F^{*}$. Let $\varphi_{b^{*}}:=\left\langle b^{*}, b^{**}\right\rangle$. For each $x\in X$, let $\delta_{(x,0)}: X^{\#}\longrightarrow \mathbb{R}$ be the linear map defined by $$\delta_{(x,0)}(f)=f(x)\ \ \  (f\in X^{\#}).$$
By setting $S_{1}x:=\delta_{(x,0)}$, $S_{2}b^{*}:=\varphi_{b^{*}}$, $G:=L_{p}(B_{X^{\#}},\mu)$, and $H:=L_{q}(B_{F^{**}},\tau)$  we obtain a Lipschitz operator $S_{1}\in\Pi_{p}^{L}(X,G)$ with $\pi_{p}^{L}(S_{1})\leq 1$ and an operator $S_{2}\in\Pi_{q}(F^{*},H)$ with $\pi_{q}(S_{2})\leq 1$ such that
\begin{align}
\left|\left\langle Tx'- Tx'', b^{*}\right\rangle\right|&\leq C\cdot\left[\int\limits_{B_{X^{\#}}}\left|f(x')-f(x'')\right|^{p} d\mu(f)\right]^\frac{1-\theta}{p} \left[\int\limits_{B_{F{**}}}\left|\left\langle b^{*}, b^{**}\right\rangle\right|^{q} d\tau(b^{**})\right]^\frac{1-\nu}{q}d_X(x',x'')^{\theta}\left\|b^{*}\right\|^{\nu} \nonumber\\
&= C\cdot d_{X}(x', x'')^{\theta}\left\|S_{1}x'-S_{1}x''| G\right\|^{1-\theta}\left\|b^{*}\right\|^{\nu}\left\|S_{2}(b^{*})|H\right\|^{1-\nu}\nonumber
\end{align}
$\bf (2)\Longrightarrow\bf (3)$ Let $x'$, $x''$ be finite sequences in $X$; $\sigma$ in $\mathbb{R}$ and $y^{*}$ be a finite sequence in $F^{*}$. By $\bf (2)$ and using the H\"older inequality with exponents of $1=\frac{1}{r}+\frac{1-\theta}{p}+\frac{1-\nu}{q}$ we have 
\begin{align}
\left[\sum\limits_{j=1}^{m}\left|\sigma_{j}\right|^{r'}\left|\left\langle Tx'_{j}- Tx''_{j}, b_{j}^{*}\right\rangle\right|^{r'}\right]^{\frac{1}{r'}}\leq C & \cdot\left[\sum\limits_{j=1}^{m}\left(\int\limits_{B_{X^{\#}}}\left(\left|\sigma_{j}\right|\left|f(x'_{j})-f(x''_{j})\right|^{1-\theta} d_X(x'_{j},x''_{j})^{\theta}\right)^\frac{p}{1-\theta} d\mu(f)\right)\right]^\frac{1-\theta}{p}\nonumber \\
&\cdot\left[\sum\limits_{j=1}^{m}\left(\int\limits_{B_{F^{**}}}\left(\left|\left\langle b_{j}^{*}, b^{**}\right\rangle\right|^{1-\nu} \left\|b_{j}^{*}\right\|^{\nu}\right)^\frac{q}{1-\nu} d\tau(b^{**})\right)\right]^\frac{1-\nu}{q} \nonumber \\
&\leq C\cdot\sup\limits_{f\in B_{{X}^{\#}}}\Bigg[\sum\limits_{j=1}^{m}\left[\left|\sigma_j\right|\left|fx'_j-fx''_j\right|^{1-\theta} d_X(x'_j,x''_j)^{\theta}\right]^{\frac{p}{1-\theta}}\Bigg]^{\frac{1-\theta}{p}}\nonumber \\
&\cdot\sup\limits_{b^{**}\in B_{{F}^{**}}}\Bigg[\sum\limits_{j=1}^{m}\left[\left|\left\langle b^{*}_j, b^{**}\right\rangle \right|^{1-\nu} \left\|b^{*}_j\right\|^{\nu}\right]^{\frac{q}{1-\nu}}\Bigg]^{\frac{1-\nu}{q}}\nonumber \\
&= C\cdot\left\|(\sigma,x',x'')\Big|\delta_{p,\theta}^{L}(\mathbb{R}\times X\times X)\right\|\left\|b^{*} \Big|\delta_{q,\nu}(F^{*})\right\| \nonumber
\end{align}
$\bf (3)\Longrightarrow\bf (2)$ Take $\left[C(B_{X^{\#}})\times C(B_{F^{**}})\right]^{*}$ equipped with the weak $C(B_{X^{\#}})\times C(B_{F^{**}})$-topology. Then $W(B_{X^{\#}})\times W(B_{F^{**}})$ is a compact convex subset. For any finite sequences $\sigma$ in $\mathbb{R}$, $x'$, $x''$ in $X$ and $y^{*}$ in $F^{*}$ the equation

\begin{align}
\Psi(\mu, \tau)=\sum\limits_{j=1}^{n}\bigg(\frac{1}{r'}\left|\sigma_{j}\left\langle Tx'_{j}-Tx''_{j}, b^{*}_{j}\right\rangle\right|^{r'}& 
-\frac{C^{r'}}{\frac{p}{1-\theta}} \int\limits_{B_{X^{\#}}}\left|\sigma_{j}\right|^{\frac{p}{1-\theta}} d_{X} (x'_j, x''_j)^{\frac{\theta p}{1-\theta}}\left|f(x'_{j})-f(x''_{j})\right|^{p} d\mu(f)\nonumber \\
& -\frac{C^{r'}}{\frac{q}{1-\nu}}\int\limits_{B_{F^{**}}} \left\|b^{*}_{j}\right\|^{\frac{\nu q}{1-\nu}}\left|\left\langle b_{j}^{*}, b^{**}\right\rangle\right|^{q} d\tau(b^{**})\bigg) \nonumber
\end{align}
defines a continuous convex function $\Psi$ on $W(B_{X^{\#}})\times W(B_{F^{**}})$. From the compactness of $B_{X^{\#}}$ and $B_{F^{**}}$, there exists $f_{0}\in B_{X^{\#}}$ and $y^{**}_{0}\in B_{F^{**}}$ such that


$$\zeta=\Bigg[\sum\limits_{j=1}^{n}\left[\left|\sigma_j\right|\left|f_0 \:x'_j-f_0 \:x''_j\right|^{1-\theta} d_X(x'_j,x''_j)^{\theta}\right]^{\frac{p}{1-\theta}}\Bigg]^{\frac{1-\theta}{p}}.$$

and

$$\beta=\Bigg[\sum\limits_{j=1}^{n}\left[\left|\left\langle b^{*}_{j},b^{**}_{0}\right\rangle \right|^{1-\nu}\left\|b^{*}_{j}\right\|^{\nu}\right]^{\frac{q}{1-\nu}}\Bigg]^{\frac{1-\nu}{q}}.$$

If $\delta (f_{0})$ and $\delta (y^{**}_{0})$ denotes the Dirac measure at the point $f_{0}$ and in $y^{**}_{0}$, respectively, then we have

\begin{align}
\Psi\left(\delta (f_{0}), \delta (b^{**}_{0})\right)=\frac{1}{r'}\sum\limits_{j=1}^{n}\bigg(\left|\sigma_j\left\langle  Tx'_{j}-Tx''_{j}, b^{*}_{j}\right\rangle\right|^{r'}& - \frac{C^{r'}}{\frac{p}{1-\theta}} \left|\sigma_{j}\right|^{\frac{p}{1-\theta}} d_{X} (x'_j, x''_j)^{\frac{\theta p}{1-\theta}}\left|f_{0}(x'_{j})-f_{0}(x''_{j})\right|^{p} \nonumber \\
& -\frac{C^{r'}}{\frac{q}{1-\nu}} \left\|b^{*}_{j}\right\|^{\frac{\nu q}{1-\nu}}\left|\left\langle b_{j}^{*}, b_{0}^{**}\right\rangle\right|^{q}\bigg) \nonumber 
\end{align}
\begin{align}
\ \ \ \ \ \: &=\frac{1}{r'}\sum\limits_{j=1}^{n}\left|\sigma_j\right|^{r'}\left|\left\langle  Tx'_{j}-Tx''_{j}, b^{*}_{j}\right\rangle\right|^{r'} - C^{r'}\left(\frac{1-\theta}{p} \zeta^{\frac{p}{1-\theta}}+\frac{1-\nu}{q} \beta^{\frac{p}{1-\theta}}\right) \nonumber \\
&\leq\frac{1}{r'}\left(\sum\limits_{j=1}^{n}\left|\sigma_j\right|^{r'}\left|\left\langle  Tx'_{j}-Tx''_{j}, b^{*}_{j}\right\rangle\right|^{r'}\right)^{r'}-\frac{C^{r'}}{r'}(\zeta\cdot\beta)^{r'} \nonumber \\
&=\frac{1}{r'}\left[\left(\sum\limits_{j=1}^{n}\left|\sigma_j\right|^{r'}\left|\left\langle  Tx'_{j}-Tx''_{j}, b^{*}_{j}\right\rangle\right|^{r'}\right)^{r'}- (C\cdot\zeta\cdot\beta)^{r'}\right] \nonumber \\
&\leq 0. \nonumber
\end{align}

Since the collection $\mathfrak{Q}$ of all functions $\Psi$ obtained in this way is concave, by \cite [E.4.2] {P78} there are $\mu_{0}\in W(B_{X^{\#}})$ and $\tau_{0}\in W(B_{F^{**}})$ such that $\Psi\left(\mu_{0}, \tau_{0} \right)\leq 0$ for all $\Psi\in\mathfrak{Q}$. In particular, if $\Psi$ is generated by single finite sequences $\sigma$ in $\mathbb{R}$, $x'$, $x''$ in $X$ and $y^{*}$ in $F^{*}$, it follows that

\begin{align}
\frac{1}{r'}\left|\sigma\left\langle Tx'-Tx'', b^{*}\right\rangle\right|^{r'}&-\frac{C^{r'}}{\frac{p}{1-\theta}}\int\limits_{B_{X^{\#}}}\left|\sigma\right|^{\frac{p}{1-\theta}} d_{X} (x', x'')^{\frac{\theta p}{1-\theta}}\left|fx'-fx''\right|^{p} d\mu_{0}(f) \nonumber \\
&-\frac{C^{r'}}{\frac{q}{1-\nu}}\int\limits_{B_{F^{**}}} \left\|b^{*}\right\|^{\frac{\nu q}{1-\nu}}\left|\left\langle b^{*}, b^{**}\right\rangle\right|^{q} d\tau_{0}(b^{**})\leq 0. \nonumber
\end{align}

Finally, we put
$$s_{1}:=\left[\int\limits_{B_{X^{\#}}}\left|\sigma\right|^{\frac{p}{1-\theta}} d_{X} (x', x'')^{\frac{\theta p}{1-\theta}}\left|fx'-fx''\right|^{p} d\mu_{0}(f)\right]^{\frac{1-\theta}{p}}$$
and
$$s_{2}:=\left[\int\limits_{B_{F^{**}}} \left\|b^{*}\right\|^{\frac{\nu q}{1-\nu}}\left|\left\langle b^{*}, b^{**}\right\rangle\right|^{q} d\tau_{0}(b^{**})\right]^{\frac{1-\nu}{q}}.$$
Then
\begin{align}
\left|\sigma\left\langle Tx'-Tx'', b^{*}\right\rangle\right|&=s_{1}s_{2}\left|(s_{1}^{-1}\sigma)\left\langle Tx'-Tx'', s_{2}^{-1} b^{*}\right\rangle\right| \nonumber \\
&\leq C s_{1} s_{2} \bigg[\frac{r'}{\frac{p}{1-\theta}}\int\limits_{B_{X^{\#}}}\left|s_{1}^{-1}\sigma\right|^{\frac{p}{1-\theta}} d_{X} (x', x'')^{\frac{\theta p}{1-\theta}}\left|fx'-fx''\right|^{p} d\mu_{0}(f) \nonumber \\
& + \frac{r'}{\frac{q}{1-\nu}}\int\limits_{B_{F^{**}}} \left\|s_{2}^{-1} b^{*}\right\|^{\frac{\nu q}{1-\nu}}\left|\left\langle s_{2}^{-1} b^{*}, b^{**}\right\rangle\right|^{q} d\tau_{0}(b^{**})\bigg]^{\frac{1}{r'}}. \nonumber \\
&\leq C \ s_{1} \ s_{2}. \nonumber
\end{align}
Hence 
\begin{align}
\left|\left\langle Tx'- Tx'', b^{*}\right\rangle\right|\leq C &\cdot\left[\int\limits_{B_{X^{\#}}}\left(\left|f(x')-f(x'')\right|^{1-\theta} d_X(x',x'')^{\theta}\right)^\frac{p}{1-\theta} d\mu(f)\right]^\frac{1-\theta}{p}\nonumber \\
&\cdot\left[\int\limits_{B_{F^{**}}}\left(\left|\left\langle b^{*}, b^{**}\right\rangle\right|^{1-\nu} \left\|b^{*}\right\|^{\nu}\right)^\frac{q}{1-\nu} d\tau(b^{**})\right]^\frac{1-\nu}{q}. \nonumber
\end{align}

\end{proof}

The aforementioned Theorem \ref{thm1} will use to prove the next result. 
\begin{cor}\label{sun2}
The linear space  $\mathcal{D}_{\left(p, \theta, q, \nu\right)}^{L}(X,F)$ is a Banach space under the norm $D_{\left(p, \theta, q, \nu\right)}^{L}$.
\end{cor}

\begin{proof} 
To prove that $\mathcal{D}_{\left(p, \theta, q, \nu\right)}^{L}(X,F)$ is complete space. We consider an arbitrary Cauchy sequence $\left(T_{n}\right)_{n\in\mathbb{N}}$ in $\mathcal{D}_{\left(p, \theta, q, \nu\right)}^{L}(X,F)$ and show that $\left(T_{n}\right)_{n\in\mathbb{N}}$ converges to $T\in\mathcal{D}_{\left(p, \theta, q, \nu\right)}^{L}(X,F)$. Since $\left(T_{n}\right)_{n\in\mathbb{N}}$ is Cauchy, for every $\epsilon>0$ there is an $n_{0}$ such that 

\begin{equation}\label{zzzzzwzzzzzz1212}
D_{\left(p, \theta, q, \nu\right)}^{L}\left(T_{m}-T_{n}\right)\leq\epsilon\ \  \text{for} \ \  m,n\geq  n_{0},
\end{equation}

Since $\Lip\left(T_{m}-T_{n}\right)\leq D_{\left(p, \theta, q, \nu\right)}^{L}\left(T_{m}-T_{n}\right)$ then $\left(T_{n}\right)_{n\in\mathbb{N}}$ is also a Cauchy sequence in the Banach space $\Lip(X,F)$, and there is a Lipschitz map $T$ with $$\lim_{n\rightarrow\infty}\Lip\left(T-T_{n}\right)=0.$$ 

From $\bf (2)$ of Theorem \ref{thm1} given $\epsilon>0$ there is  $n_{0}\in\mathbb{N}$ such that, for each $n, m\in\mathbb{N}$, $n, m\geq n_{0}$, there exist   probabilities $\mu_{n m}$ on $B_{X^{\#}}$ and $\tau_{n m}$  on $B_{F^{**}}$ such that for every $x'$, $x''$ in $X$ and $b^{*}$ in $F^{*}$ the following inequality holds
\begin{align}
\left|\left\langle (T_{m}-T_{n}) \; x'- (T_{m}-T_{n}) \; x'', b^{*}\right\rangle\right|\leq  \epsilon \; d_X(x',x'')^{\theta} \; \left\|b^{*}\right\|^{\nu} &\cdot\left[\int\limits_{B_{X^{\#}}} \left|f(x')-f(x'')\right|^{p}   d\mu_{n m}(f)\right]^{1-\theta} \nonumber \\
&\cdot\left[\int\limits_{B_{F^{**}}} \left|\left\langle b^{*}, b^{**}\right\rangle\right|^{q}   d\tau_{n m}(b^{**})\right]^ {1-\nu} . \nonumber
\end{align}
Fixed $n\geq  n_{0}$, by the weak compactness of $W\left(B_{X^{\#}}\right)$ and $W\left(B_{F^{**}}\right)$, there is a sub-net $\left(\mu_{n m}(\alpha), \tau_{n m}(\alpha)\right)_{\alpha\in\mathcal{A}}$ convergent to $\left(\mu_{n }, \tau_{n }\right)\in W(B_{X^{\#}})\times W(B_{F^{**}})$ in the topology $\sigma\left((C(B_{X^{\#}}\times C(B_{F^{**}})))^{\ast}, C(B_{X^{\#}})\times C(B_{F^{**}})\right)$. Then, there is $\alpha_{0}\in\mathcal{A}$ such that for each $x'$, $x''$ in $X$, $b^{*}$ in $F^{*}$ and $\alpha\in\mathcal{A}$ with $\alpha\geq\alpha_{0}$ we have
\begin{align}
&\left|\left\langle (T_{m(\alpha)}-T_{n}) \; x'- (T_{m(\alpha)}-T_{n}) \; x'', b^{*}\right\rangle\right|\leq  \epsilon \; d_X(x',x'')^{\theta} \; \left\|b^{*}\right\|^{\nu} \nonumber \\
&\cdot\left[\int\limits_{B_{X^{\#}}} \left|f(x')-f(x'')\right|^{p}   d(\mu_{{n m}(\alpha)}-\mu_{n})(f) + \int\limits_{B_{X^{\#}}} \left|f(x')-f(x'')\right|^{p} d\mu_{n}(f)\right]^{1-\theta} \nonumber \\
&\cdot\left[\int\limits_{B_{F^{**}}} \left|\left\langle b^{*}, b^{**}\right\rangle\right|^{q}   d(\tau_{{n m}(\alpha)}-\tau_{n})(b^{**}) + \int\limits_{B_{F^{**}}} \left|\left\langle b^{*}, b^{**}\right\rangle\right|^{q}   d\tau_{n} (b^{**})\right]^ {1-\nu} . \nonumber
\end{align}
and  taking limits when $\alpha\in\mathcal{A}$ we have 

\begin{align} 
\left|\left\langle (T -T_{n}) \; x'- (T -T_{n}) \; x'', b^{*}\right\rangle\right|\leq  \epsilon \; d_X(x',x'')^{\theta} \; \left\|b^{*}\right\|^{\nu} &\cdot\left[\int\limits_{B_{X^{\#}}} \left|f(x')-f(x'')\right|^{p}   d\mu_{n}(f)\right]^{1-\theta} \nonumber \\
&\cdot\left[\int\limits_{B_{F^{**}}} \left|\left\langle b^{*}, b^{**}\right\rangle\right|^{q}   d\tau_{n}(b^{**})\right]^{1-\nu}  \nonumber 
\end{align}
for every $x'$, $x''$ in $X$ and $b^{*}$ in $F^{*}$. It follows that  $T -T_{n}\in\mathcal{D}_{\left(p, \theta, q, \nu\right)}^{L}(X,F)$ and therefore  $T\in\mathcal{D}_{\left(p, \theta, q, \nu\right)}^{L}(X,F)$. From the last inequality it follows that $D_{\left(p, \theta, q, \nu\right)}^{L}(T -T_{n})\leq\epsilon$ if $n\geq  n_{0}$ and hence $\mathcal{D}_{\left(p, \theta, q, \nu\right)}^{L}(X,F)$ is a Banach space.
\end{proof} 

By Proposition \ref{sun1}, Theorem \ref{thm1}, and Corollary  \ref{sun2} we obtain the following result.

\begin{prop} 
$\left[\mathcal{D}_{\left(p, \theta, q, \nu\right)}^{L}, D_{\left(p, \theta, q, \nu\right)}^{L}\right]$ is a Banach nonlinear ideal.
\end{prop} 

\begin{remark}
 Definition \ref{Torezen} can be generalized as follows. Let $0\leq \theta< 1$ and  $\left[\textfrak{A}^{L}, \mathbf{A}^{L}\right]$ be a normed nonlinear ideal. A Lipschitz operator $T$ from $X$ into $F$ belongs to $\left(\textfrak{A}^{L}, \textfrak{B}^{L}\right)_{\theta}(X, F)$ if there exist a Banach spaces $G_{1}$, $G_{2}$ and a Lipschitz operators $S_{1}\in\textfrak{A}^{L}(X, G_{1})$ and $S_{2}\in\textfrak{B}^{L}(X, G_{2})$ such that 
\begin{equation}\label{lastwagen}
\left\|Tx'-Tx''|F\right\|\leq\left\|S_{1}x'-S_{1}x''|G_{1}\right\|^{1-\theta}\cdot \left\|S_{2}x'-S_{2}x''|G_{2}\right\|^{\theta},\ \  \forall\: x',\; x'' \in X.
\end{equation}
For each $T\in\left(\textfrak{A}^{L}, \textfrak{B}^{L}\right)_{\theta}(X, F)$, we set 
\begin{equation}
\left(\mathbf{A}^{L}, \mathbf{B}^{L}\right)_{\theta}(T):=\inf\mathbf{A}^{L}(S_{1})^{1-\theta}\cdot\mathbf{B}^{L}(S_{2})^{\theta}
\end{equation}
where the infimum is taken over all Lipschitz operators $S_{1}$, $S_{2}$  admitted in (\ref{lastwagen}). 

Note that $\Lip(T)\leq\left(\mathbf{A}^{L}, \mathbf{B}^{L}\right)_{\theta}(T)$, by definition. The nonlinear ideal $\left[\textfrak{A}^{L}_{\theta}, \mathbf{A}^{L}_{\theta}\right]$ now appear as $\left[\left(\textfrak{A}^{L}, \Lip\right)_{\theta}, \left(\mathbf{A}^{L}, \Lip(\cdot)\right)_{\theta}\right]$.
\end{remark}

\section{\bf Nonlinear operator ideals between metric spaces}\label{Sec. 4}

\begin{definition}\label{Auto777777777}
Suppose that, for every pair of metric spaces $X$ and $Y$, we are given a subset $\mathscr{A}^{L}(X,Y)$ of $\mathscr{L}(X,Y)$. The class  
$$\mathscr{A}^{L}:=\bigcup_{X,Y}\mathscr{A}^{L}(X,Y)$$
is said to be a nonlinear operator ideal, or just a nonlinear ideal, if the following conditions are satisfied:

\begin{enumerate}
	\item[$\bf (\widetilde{\widetilde{NOI_0\,}})$] If $Y=F$ is a Banach space, then  $g\boxdot e\in\mathscr{A}^{L}(X,F)$ for $g\in X^{\#}$ and $e\in F$.
	\item[$\bf (\widetilde{\widetilde{NOI_1\,}})$] $BTA\in\mathscr{A}^{L}(X_{0}, Y_{0})$ for $A\in \mathscr{L}(X_{0},X)$, $T\in\mathscr{A}^{L}(X, Y)$, and $B\in\mathscr{L}(Y, Y_{0})$.
\end{enumerate}
Condition $\bf (\widetilde{\widetilde{NOI_0}})$ implies that $\mathscr{A}^{L}$ contains nonzero Lipschitz operators.
\end{definition}

\subsection{Lipschitz interpolative ideal procedure between metric spaces}\label{Au3to20}

\begin{definition}\label{fruher}
Let $0\leq \theta< 1$. A Lipschitz map $T$ from $X$ into $Y$ belongs to $\mathscr{A}^{L}_{\theta}(X, Y)$ if there exist a constant $C\geq 0$, a metric space $Z$ and a Lipschitz map $S\in\mathscr{A}^{L}(X, Z)$ such that 
\begin{equation}\label{poooor}
d_{Y}(Tx', Tx'')\leq C \cdot d_{Z}(Sx', Sx'')^{1-\theta}\cdot d_{X}(x',x'')^{\theta},\ \  \forall\: x',\; x'' \in X.
\neq\end{equation}
For each $T\in\mathscr{A}^{L}_{\theta}(X, Y)$, we set 
\begin{equation}\label{fliehen}
\mathbf{A}^{L}_{\theta}(T):=\inf C \mathbf{A}^{L}(S)^{1-\theta}
\end{equation}
where the infimum is taken over all Lipschitz operators $S$ admitted in (\ref{poooor}). 
\end{definition} 	

\begin{prop}
$\mathscr{A}^{L}_{\theta}$ is a nonlinear ideal with $\mathbf{A}^{L}_{\theta}(BTA)\leq\Lip(B)\cdot\mathbf{A}^{L}_{\theta}(T)\cdot\Lip(A)$ for $A\in \mathscr{L}(X_{0},X)$, $T\in\mathscr{A}^{L}(X, Y)$, and $B\in\mathscr{L}(Y, Y_{0})$.. 
\end{prop}

\begin{proof}
The proof of condition $\bf (\widetilde{\widetilde{NOI_0\,}})$ is similar to the proof of algebraic condition $\bf (\widetilde{PNOI_0})$ in Proposition \ref{Flughafen}. To prove the condition $\bf (\widetilde{\widetilde{NOI_1\,}})$, let $A\in \mathscr{L}(X_{0},X)$, $T\in\mathscr{A}^{L}_{\theta}(X, Y)$, and $B\in\mathscr{L}(Y, Y_{0})$ and $x'_{0}, x''_{0}$ in $X$. Then
\begin{align}
d_{Y_{0}}(BTA x'_{0}, BTA x''_{0})&\leq \Lip(B)\cdot d_{Y}(TA x'_{0}, TA x''_{0}) \nonumber \\
&\leq\Lip(B)\cdot C\cdot d_{Z}(S\circ A (x'), S\circ A (x''))^{1-\theta}\cdot d_{X}(Ax',Ax'')^{\theta} \nonumber \\
&\leq C \cdot\Lip(B)\cdot\Lip(A)^{\theta}\cdot d_{Z}(S\circ A (x'), S\circ A (x''))^{1-\theta} d_{X}(x',x'')^{\theta}. \nonumber
\end{align}
Since a Lipschitz map $\widetilde{S}:=S\circ A\in\mathscr{A}^{L}(X_{0}, Z)$ and $\widetilde{C}:=C \cdot\Lip(B)\cdot\Lip(A)^{\theta}$, hence $BTA\in\mathscr{A}^{L}_{\theta}(X_{0}, Y_{0})$. Moreover, from (\ref{fliehen}) we have
\begin{align}\label{jo}
\mathbf{A}^{L}_{\theta}(BTA)&:=\inf \widetilde{C}\cdot \mathbf{A}^{L}(\widetilde{S})^{1-\theta} \nonumber \\
&\leq C \cdot\Lip(B)\cdot\Lip(A)^{\theta}\cdot \mathbf{A}^{L}(S\circ A)^{1-\theta} \nonumber \\
&= C \cdot\Lip(B)\cdot\Lip(A)\cdot \mathbf{A}^{L}(S)^{1-\theta}. 
\end{align}
Taking the infimum over all such $S\in\mathscr{A}^{L}(X, Z)$ on the right side of (\ref{jo}), we have
$$\mathbf{A}^{L}_{\theta}(BTA)\leq\Lip(B)\cdot\mathbf{A}^{L}_{\theta}(T)\cdot\Lip(A).$$
\end{proof}

\subsection{Basic Examples of Lipschitz interpolative ideal procedure}

\begin{enumerate}
	\item[$\bf (1)$] Lipschitz $\left(p,s,\theta\right)$-summing maps\label{schlumpf02}
	
	A Lipschitz map $T$ from $X$ to $Y$ is called Lipschitz $\left(p,s,\theta\right)$-summing if there is a constant $C\geq 0$  such that 
	
\begin{equation}\label{metaphor}
\left\|(\sigma,Tx',Tx'')\Big|\ell_{\frac{p}{1-\theta}}(\mathbb{R}\times Y\times Y)\right\|\leq C\cdot\left\|(\sigma,x',x'')\Big|\delta_{s,\theta}^{L}(\mathbb{R}\times X\times X)\right\|.
\end{equation}	
	
 for arbitrary finite sequences $x'$, $x''$ in $X$ and $\sigma$ in $\mathbb{R}$. Let us denote by $\Pi^{L}_{\left(p,s,\theta\right)}(X,Y)$ the class of all Lipschitz $\left(p,s,\theta\right)$-summing maps from $X$ to $Y$ with $$\pi^{L}_{\left(p,s,\theta\right)}(T)=\inf C.$$
where the infimum is taken over all constant $C$ satisfying (\ref{metaphor}). The proof of the following result is not difficult to prove it.

\begin{prop}\label{schlumrtpf02}
$\left[\Pi^{L}_{\left(p,s,\theta\right)}, \pi^{L}_{\left(p,s,\theta\right)}\right]$ is a  nonlinear ideal. 
\end{prop}

\begin{remark}

\item[$\bf (1)$] If $\theta = 0$, then the class  $\Pi^{L}_{\left(p,s,\theta\right)}(X,Y)$ coincides with the class $\Pi^{L}_{\left(p,s\right)}(X,Y)$ which considered in \cite {Mass15} for $\infty\geq p\geq q > 0$ and \cite{JA12} for $1\leq q< p$.

\item[$\bf (2)$] For a special case, if $p=s$,   a Lipschitz $\left(p,\theta\right)$-summing map defined in \cite{AcRuYa} if there is a constant $C\geq 0$  such that 
\begin{equation}\label{metaphor1}
\left\|(\lambda,Tx',Tx'')\Big|\ell_{\frac{p}{1-\theta}}(\mathbb{R}\times Y\times Y)\right\|\leq C\cdot\left\|(\lambda,x',x'')\Big|\delta_{p,\theta}^{L}(\mathbb{R}\times X\times X)\right\|.
\end{equation}
for arbitrary finite sequences $x'$, $x''$ in $X$ and $\lambda$ in $\mathbb{R^{+}}$. Let us denote by $\Pi^{L}_{\left(p,\theta\right)}(X,Y)$ the class of all Lipschitz $\left(p,\theta\right)$-summing maps from $X$ to $Y$ with $$\pi^{L}_{\left(p,\theta\right)}(T)=\inf C.$$
\end{remark}
where the infimum is taken over all constant $C$ satisfying (\ref{metaphor1}). The next result is a consequence of Proposition \ref{schlumrtpf02}.
\begin{cor}
$\left[\Pi^{L}_{\left(p,\theta\right)}, \pi^{L}_{\left(p,\theta\right)}\right]$ is a   nonlinear ideal.  
\end{cor}

As a consequence of a general definition of Lipschitz interpolative ideal procedure between metric spaces   the Lipschitz  $(p, \theta)$-summing map    has  following characterize result.
   
\begin{thm} \cite{AcRuYa}
\label{dom} Let $1\leq p<\infty $, $0\leq \theta <1$ and $T\in \mathrm{Lip}%
(X,Y)$. \textit{The following statements are equivalent.}

\begin{enumerate}
\item[(i)] $T\in \Pi _{p,\theta }^{L}(X,Y)$.

\item[(ii)] \textit{There is a constant }$C\geq 0$ and a regular Borel
probability measure $\mu $ on $B_{X^{\#}}$ such that
\begin{equation*}
d_{Y}(Tx', Tx'')\leq C\left( \int\nolimits_{B_{X^{\#}}}\left(
|fx'- fx''|^{1-\theta } d_{X}(x', x'')^{\theta }\right) ^{\frac{p%
}{1-\theta }} d\mu \left( f\right) \right) ^{\frac{1-\theta }{p}}
\end{equation*}%
for all $x', x''\in X$.

\item[(iii)] There is a constant $C\geq 0$ such that for all $%
(x'_{i})_{j=1}^{m},(x''_{j})_{j=1}^{m}$ in $X$ and all $%
(a_{j})_{j=1}^{m}\subset \mathbb{R}^{+}$ we have
\begin{equation*}
\begin{array}{l}
\displaystyle\hspace{-1cm}\left(
\sum_{j=1}^{m}a_{i} d_{Y}(T(x'_{j}),T(x''_{j}))^{\frac{p}{1-\theta }%
}\right) ^{\frac{1-\theta }{p}} \\
\displaystyle\leq C\underset{f\in B_{X^{\#}}}{\sup }\left(
\sum_{j=1}^{m} a_{j}\left( |f(x'_{j})-f(x''_{j})|^{1-\theta
}d_{X}(x'_{j}, x''_{j})^{\theta }\right) ^{\frac{p}{1-\theta }}\right) ^{%
\frac{1-\theta }{p}}\text{.}%
\end{array}%
\end{equation*}

\item [(iv)] There exists a regular Borel probability measure $\mu $\textit{on }$B_{X^{\#}}$ \textit{and a Lipschitz operator } $v:X_{p,\theta }^{\mu
}\rightarrow Y$\textit{\ such that the following diagram commutes}
\begin{equation*}
\xymatrix{X\ar[r]^T\ar[d]^{\delta_{X}} & Y\\ \delta_{X}(X) \ar[r]^{\phi\circ i}
&X_{p,\theta }^\mu \ar[u]^v}
\end{equation*}%
Furthermore, the infimum of the constants $C\geq 0$ in $(2)$ and $(3)$ is $%
\pi _{p,\theta }^{L}\left( T\right) $.
\end{enumerate}
\end{thm}


\item[$\bf (2)$] Lipschitz $ \left( s;q,\theta \right) $-mixing operators\label{schl45umpf02} 

D. Achour, E. Dahia and M. A. S. Saleh \cite{aem18} defined a Lipschitz $ \left( s;q,\theta \right) $-mixing operator if there is a constant $C\geq 0$ such that
\begin{equation}
\mathfrak{m}_{(s;q)}^{L,\theta }(\sigma ,Tx^{\prime },Tx^{\prime \prime
})\leq C\cdot \delta _{q\theta }^{L}(\sigma ,x^{\prime },x^{\prime \prime })
\label{flat1}
\end{equation}%
for arbitrary finite sequences $x^{\prime }$, $x^{\prime \prime }$ in $X$ and $\sigma $ in $\mathbb{R}$. Let us denote by $\mathbf{M}_{(s;q)}^{L,\theta }(X,Y)$ the class of all Lipschitz $\left( s;q,\theta
\right) $-mixing maps from $X$ to $Y.$ In such case, we put
\begin{equation*}
\mathbf{m}_{(s;q)}^{L,\theta }(T)=\inf C,
\end{equation*}%
where the infimum is taken over all constant $C$ satisfying (\ref{flat1}). The proof of the following result is not difficult to prove it.

\begin{prop}\label{schwwlumrtpf02}
$\left[\mathbf{M}_{(s;q)}^{L,\theta }, \mathbf{m}_{(s;q)}^{L,\theta }\right]$ is a nonlinear ideal. 
\end{prop}

\end{enumerate}

\begin{remark}
 Definition \ref{fruher} can be generalized as follows. Let $0\leq \theta< 1$. A Lipschitz map $T$ from $X$ into $Y$ belongs to $\left(\mathscr{A}^{L}, \mathscr{B}^{L}\right)_{\theta}(X, Y)$ if there exist a constant $C\geq 0$, a metric spaces $Z_{1}$, $Z_{2}$ and a Lipschitz maps $S_{1}\in\mathscr{A}^{L}(X, Z_{1})$ and $S_{2}\in\mathscr{B}^{L}(X, Z_{2})$ such that 
\begin{equation}\label{qlastwagen}
d_{Y}(Tx', Tx'')\leq C\cdot d_{Z_{1}}(S_{1}x', S_{1}x'')^{1-\theta}\cdot d_{Z_{2}}(S_{2}x', S_{2}x'')^{\theta},\ \  \forall\: x',\; x'' \in X.
\end{equation}
For each $T\in\left(\mathscr{A}^{L}, \mathscr{B}^{L}\right)_{\theta}(X, Y)$, we set 
\begin{equation}
\left(\mathbf{A}^{L}, \mathbf{B}^{L}\right)_{\theta}(T):=\inf C\cdot\mathbf{A}^{L}(S_{1})^{1-\theta}\cdot\mathbf{B}^{L}(S_{2})^{\theta}
\end{equation}
where the infimum is taken over all Lipschitz operators $S_{1}$, $S_{2}$  admitted in (\ref{qlastwagen}). The nonlinear ideal $\left[\mathscr{A}^{L}_{\theta}, \mathbf{A}^{L}_{\theta}\right]$ now appear as $\left[\left(\mathscr{A}^{L}, \mathscr{L}\right)_{\theta}, \left(\mathbf{A}^{L}, \Lip(\cdot)\right)_{\theta}\right]$. 

\end{remark}

\end{document}